\documentclass[a4paper]{article}
\usepackage{amsmath,amssymb,a4wide}
\usepackage{color,graphicx,epstopdf}
\usepackage[a4paper]{geometry}

\newcommand{\bu}{\boldsymbol u}
\newcommand{\bv}{\boldsymbol v}

\newcommand{\be}{\boldsymbol e}

\newcommand{\bU}{\boldsymbol U}

\newcommand{\di}{\mathrm d}

\newtheorem{Theorem}{Theorem}
\newtheorem{lema}{Lemma}
\newcounter{remark}
\def\theremark {\arabic{remark}}
\newenvironment{remark}{\refstepcounter{remark}\par\noindent{\bf Remark\ \theremark}\ }{\par}
\newtheorem{Proof}{Proof}

\newenvironment{proof}{\begin{Proof}\rm}{\hfill $\Box$ \end{Proof}}

\usepackage{hyperref}

\title{POD-ROM methods: error analysis for continuous parametrized approximations}
\author{Bosco
Garc\'{\i}a-Archilla\thanks{Departamento de Matem\'atica Aplicada
II, Universidad de Sevilla, Sevilla, Spain. Research is supported by
Spanish MCINYU under grants PID2021-123200NB-I00 and PID2022-136550NB-I00. (bosco@esi.us.es)}
\and Alicia Garc\'{\i}a-Mascaraque\thanks{Departamento de
Matem\'aticas, Universidad Aut\'onoma de Madrid, Spain. Research is supported
by predoctoral contract associated to PID2022-136550NB-I00 supported by MCIN/AEI/10.13039/501100011033 and FSE+ (alicia.garcia-mascaraque@uam.es)}
  \and Julia Novo\thanks{Departamento de
Matem\'aticas, Universidad Aut\'onoma de Madrid, Spain. Research is supported
by Spanish MINECO
under grant PID2022-136550NB-I00. (julia.novo@uam.es)}}
\date{\today}

\begin{document}
\maketitle
\begin{abstract}This paper studies the numerical approximation of parametric time-dependent partial differential equations (PDEs)
by proper orthogonal decomposition reduced order models (POD-ROMs). Although many papers in the literature consider reduced order models for parametric equations, a complete error analysis of the methods is still a challenge. We introduce and
analyze in this paper a new POD method based on finite differences (respect to time and respect to the parameters that may be considered). 
We obtain a priori bounds for the new method valid for any value of time in a given time interval and any value of the parameter in a given parameter interval. 
Our design of the new POD method allow us to prove pointwise-in-time error estimates as opposed to average error bounds obtained typically in POD methods.  Most of the papers concerning POD methods for parametric equations are just based on the snapshots computed at different times and parameter values instead of their difference quotients. We show that the error analysis of the present paper can also cover the error analysis of that case (that we call standard). Some numerical experiments compare our new approach with the standard one and support the error analysis.
\end{abstract}
\bigskip


\section{Introduction}
Reduced-order methods can significantly reduce the computational cost required to obtain numerical approximations while also trying to achieve sufficient accuracy. A
% very interesting
frequent
  case is that of equations that depend on one or several parameters. The challenge in this case is to provide accurate approximations for parameter values that are not part of the dataset used to compute the reduced-order basis. In this paper we consider a type of reduced order methods that are called proper orthogonal decomposition (POD) methods, \cite{ku-Vol}.

As stated in \cite{mamo_ol}, although numerical analysis of traditional discretization methods for PDEs is a mature research field, error analysis of reduced order models is still a challenging task. In particular, the complete error analysis of POD methods for evolutionary partial differential equations is a topic that remains partially open.

Most of the papers concerning error analysis of POD methods for parametrized equations provide error bounds for reduced approximations corresponding to snapshots used as part of the dataset, \cite{gun_siam}, \cite{gun_ima}. This is of course interesting and can be considered the starting point to address the problem of approaching those equations for values of the parameters that are out of the dataset. See also the error analysis of \cite{vol} for POD methods applied to parametric elliptic PDEs. As stated in \cite{mamo_ol}, an important development in an attempt to get bounds for out-of-sample parameters is the error analysis of POD-greedy algorithms measured in terms of Kolmogorov $n$-widths \cite{cohen_etal}, \cite{buffa_etal}, \cite{haas}. In this paper, we do not follow this type of approach. 

It is not our intention to cover all the related references concerning reduced methods for parametric equations. However, before going into detail about the results of the present paper we want to mention closely-related results obtained in the recent paper \cite{mamo_ol} for a different method. The authors of \cite{mamo_ol} consider a tensor ROM method for parameter dependent linear parabolic equations.  In their method, they replace in the parameter setting singular value decomposition by a low-rank tensor decomposition (LRTD). They construct parameter specific reduced spaces to approach solutions for out-of-sample parameters.
The time discretization of the reduced model  is based on the implicit Euler method. A priori error bounds for the tensor ROM method are provided in \cite{mamo_ol}. These bounds are first order in time and depend on the level of accuracy of the low-rank approximation used and on the point density in the parameter space.

In the present paper we follow the approach of the recent reference \cite{temporal_nos}. In \cite{temporal_nos}, the numerical
approximation of semilinear reaction-diffusion equations by POD methods is considered. The authors analyze the continuous in time case for both the FOM and POD methods although the snapshots are based on a discrete set of times. 
%Two cases are considered in \cite{temporal_nos}: in the first one, the POD method is obtained from difference quotients and, in the %second, from time derivatives of the FOM. 
Pointwise-in-time error estimates are proved in \cite{temporal_nos} for any value of the time variable in the time interval in which the POD approximation is computed. The present paper could be considered in some sense an extension of the results in \cite{temporal_nos} to parameter dependent equations. However, the difficulty of the extension has given rise to new different ideas which we believe are intrinsically interesting.

As in \cite{temporal_nos},
we consider as a model problem semilinear reaction-diffusion equations but our model problem is parameter dependent. The parameter (or parameters) can appear in the diffusion coefficient, nonlinear term, forcing term and/or initial condition. As in \cite{temporal_nos}, we consider the continuous in time case for both the FOM and the POD method to show the error bounds that can be obtained using any time integrator. For the POD method we consider $H_0^1$ projections since, as explained in \cite{rubino_etal}, optimal error bounds in terms of the tail of the eigenvalues can be obtained in this way. By tail of the eigenvalues we mean
the quatintity
\begin{equation}
\label{the_sigma_r}
\Sigma_r^2\equiv \sum_{k>r} \lambda_k
\end{equation}
where $\lambda_1\ge \lambda_2 \ge \ldots \ge \lambda_{d_r}>0$ are nonzero eigenvalues of the correlation matrix of the data set in the POD method (see Section~\ref{sec3} below).
 Our first approach considers a model problem depending on only one parameter (apart from time). We compute snapshots with the FOM for different values of the parameter at different times. Then, we design a POD method that is based on the values of the snapshots at the initial time (for all the different values of the parameters), on different quotients in time (for the first value of the parameter set) and on second order difference quotients (first order in the parameter followed by first order in time) for the rest of values of times and parameters. 
Then, for any value of the parameter in the parameter set (including out-of-sample parameters) and a given time interval, we obtain a priori bounds for the new method for any value of time in the given interval. 
The error bounds depend on the quantity~$\Sigma_r$ in~\eqref{the_sigma_r}, on the distance between two consecutive values of time where the snapshots are taken, and on the distance between two consecutive values of the parameter. The new method is designed in such a way that pointwise-in-time error estimates for any value of time can be proved (as opposed to average error bounds obtained in many POD methods). 
In a second part of the paper we consider the extension of the error analysis to the case of a model depending on two parameters (apart from time). In that case, the method is based on first, second and third order difference quotients. 

Most of the papers concerning POD methods for parametric equations are based on the snapshots computed at different times and parameter values instead of difference quotients, as the method we propose in this paper. Using the ideas in \cite{Bosco_pointwise}, we will show that the error analysis of this paper can be adapted to cover also the error analysis of the standard method. Some numerical experiments are presented to compare our new approach with the standard one and to support the error analysis of this paper.
 
 The outline of the paper is as follows. In Section \ref{sec2} we state some preliminaries and notation. Section \ref{sec3} is devoted to present the POD method and to prove some auxiliary results that will be needed for the error analysis of the method in Section \ref{sec4}. 
 In Section \ref{sec5} we consider the error analysis of the standard method. In Section \ref{sec6} we comment on the extension of the error analysis to the case of two parameters. Finally, in Section \ref{sec7} we show some numerical experiments.

\section{Preliminaries and notation}\label{sec2}
As a model problem we consider the following reaction-diffusion equation 
\begin{equation}\label{eq:model}
\begin{array}{rcll}
u_t^\alpha(t,x)-\nu_\alpha\Delta u^\alpha(t,x)+g_\alpha(u^\alpha(t,x))& = &f_\alpha(t,x) &\quad (t,x)\in (0,T^\alpha]\times \Omega,\\
u^\alpha(t,x)&=&0,& \quad  (t,x)\in (0,T^\alpha]\times \partial \Omega,\\
u^\alpha(0,x)&=&u_0^\alpha(x),&\quad  x\in \Omega,
\end{array}
\end{equation}
in a bounded domain $\Omega\subset {\Bbb R}^d$, $d\in \{1,2,3\}$, where $\alpha$ is a parameter, $\alpha\in [\alpha_0,\alpha_L]$, $\nu_\alpha>0$ is the diffusion parameter and $g_\alpha$ is a nonlinear smooth function. The diffusion parameter, $\nu_\alpha$, nonlinear
term, $g_\alpha$, forcing term, $f_\alpha$, and initial condition, $u_0^\alpha$, may or may not depend all of them on the paratemer 
$\alpha$, but there is at least 
one of them that depends on $\alpha$. 
%Of course, the solution $u$ of \eqref{eq:model} depends on $\alpha$, however, we will eliminate the dependency on $\alpha$ in the %notation of $u$ whenever there is no room for confusion.

{For simplicity here and in the sequel we have taken $t_0=0$ as initial time}.

Let $C_p$ be the constant in the Poincar\'e inequality
\begin{equation}\label{poincare}
\|v\|_0\le C_p \|\nabla v\|_0,\quad v\in H_0^1(\Omega).
\end{equation}
        Let us denote by $X_h^k$ a finite element space based on {continuous piecewise} polynomials of degree $k$  defined
over a partition of $\Omega$ into simplices of diameter $h$ 
and by $V_h^k$ the  finite element space based on continuous piecewise  polynomials of degree $k$ that satisfies also the homogeneous Dirichlet boundary conditions on the boundary $\partial \Omega$.

In the sequel,  $I_h u \in X_h^k$ will denote
the Lagrange interpolant of a continuous function $u$. The following bound can be found in \cite[Theorem 4.4.4]{brenner-scot}
\begin{equation}\label{cota_inter}
|u-I_h u|_{W^{m,p}(K)}\le c_\text{\rm int} h^{n-m}|u|_{W^{n,p}(K)},\quad 0\le m\le n\le k+1,
\end{equation}
where $n>d/p$ when $1< p\le \infty$ and $n\ge d$ when $p=1$. {We will also use the following bound (see  \cite[Corollary 4.4.7]{brenner-scot})
\begin{equation}\label{cota_inter_inf}
|u-I_h u|_{L^{\infty}(K)}\le c_\text{\rm int} h^{n-d/2}|u|_{H^{n}(K)},\quad 0\le n\le k+1,
\end{equation}
where $n>d/2$.}

We  consider a smooth curve $\alpha\in [\alpha_0,\alpha_L]\mapsto T^\alpha\in(0,T]$, for $\alpha_0, \alpha_L$ given parameters.
Our aim is to approximate the solutions of \eqref{eq:model} on given time intervals $[0,T^\alpha]$ 
for parameters $\alpha$ in the given set $[\alpha_0,\alpha_L]$. 

Let us consider the semi-discrete finite element approximation: Find $u_h^\alpha\ :\ (0,T^\alpha]\rightarrow V_h^k$ such that
\begin{equation}\label{gal_semi}
(u_{h,t}^\alpha,v_h)+\nu_\alpha(\nabla u_h^\alpha,\nabla v_h)+(g_\alpha(u_h^\alpha),v_h)=(f_\alpha,v_h),\quad \forall\ v_h\in V_h^k,
\end{equation}
with $u_h^{\alpha}(0)=I_h u_0^\alpha\in V_h^k.$ 
If the weak solution of \eqref{eq:model} is sufficiently smooth, then the following error estimate is well known,
see for example \cite[Proof of Theorem 1]{bosco-titi-fem}, \cite[Theorem 14.1]{Thomee}, \cite[Lemma 4.2]{Wang}, 
\begin{equation}\label{cota_gal}
\max_{0\le t\le T}\left(\|(u^\alpha-u_h^{\alpha})(t)\|_0+h\|(u^\alpha-u_h^{\alpha})(t)\|_1\right)\le C(u)h^{k+1},\quad \alpha\in[\alpha_0,\alpha_L],\ 0\le t\le T^\alpha.
\end{equation}

\section{Proper orthogonal decomposition}\label{sec3}

In this section we describe our new method based on POD.

Let us fix $L>0$, and two values $\alpha_0<\alpha_L$ of the parameter.  Let us denote $\Delta \alpha=(\alpha_L-\alpha_0)/L$ and $\alpha_l=\alpha_0+l\Delta \alpha$, $l=0,\ldots,L$. 

For each $l=0,\ldots,L$ and final time $T^l=T^{\alpha_l}>0$ we define $\Delta t_l=T^l/M$, for a fixed integer $M>0$. Let $t_j^l=j\Delta t_l$, $j=0,\ldots,M$. Let us observe that we always take $M+1$ points in each time interval, even though the intervals may have different lengths.

We will denote by $D^t$ the finite difference respect to time and by $D^\alpha$ the finite difference respect to $\alpha$. This means that
$$
D^t v^{\alpha_l}(t_j^l)=\frac{v^{\alpha_l}(t_j^l)-v^{\alpha_l}(t_{j-1}^l)}{\Delta t_l},\ 
D^\alpha v^{\alpha_l}(t_j^l)= \frac{v^{\alpha_l}(t_j^l)-v^{\alpha_{l-1}}(t_{j}^{l-1})}{\Delta \alpha},\ 1\le j\le M,\ 1\le l\le L,
$$
where $v^{\alpha_l}:[0,T^l]\rightarrow V_h^k$, $l=0,\ldots,L$. {Notice that, since the values of $T^l$ are not necessarily equal, the values of~$t_j^l$ and $t_j^{l-1}$ may not necessarily coincide.}

Let us denote by $N=(M+1)(L+1)$. We define the space
$
\bU= {\rm span}\left( {\cal U}\right),
$
where
\begin{eqnarray*}
 {\cal U} &=&\left\{\sqrt{N}u_h^{\alpha_l}(t_0), \ 0\le l\le L,\right.\\
&&\sqrt{(L+1)}D^t u_h^{\alpha_0}(t_j^0),\ 1\le j\le M,\\
&&\left. D^t D^\alpha u_h^{\alpha_j}(t_j^l),\ 1\le j\le M,\ 1\le l\le L\right\}.
\end{eqnarray*}
Let us observe that the set ${\cal U}$ has $N$ elements. Keeping the order  of the elements in ${\cal U}$ we can write
\begin{equation}
\label{calU2}
{\cal U}=\left\{y_h^1,\ldots,y_h^N\right\}.
\end{equation}
Let 
$X=H_0^1(\Omega)$ and let us denote the correlation matrix by $S=((s_{i,j}))\in {\Bbb R}^{N\times N}$ with
$$
s_{i,j}=\frac{1}{N}(y_h^i,y_h^j)_X=\frac{1}{N}(\nabla y_h^i,\nabla y_h^j),\quad i,j=1,\ldots,N,
$$
and $(\cdot,\cdot)$ the inner product in $L^2(\Omega)$. We denote by $\lambda_1\ge \lambda_2\ldots\ge \lambda_{d_r}>0$ the positive eigenvalues of $S$ and
by $\bv_1,\ldots,\bv_{d_r}\in {\Bbb R}^N$ the associated eigenvectors. The orthonormal POD basis functions of $\bU$ are
\begin{equation}\label{eq:varphi}
\varphi_k=\frac{1}{\sqrt{N}}\frac{1}{\sqrt{\lambda_k}}\sum_{j=1}^Nv_k^j y_h^j,
\end{equation}
where $v_k^j$ is the $j$ component of the eigenvector $\bv_k$.
For any $1\le r\le d_r$ let 
\begin{equation}\label{eq:bU_r}
\bU^r={\rm span}\left\{\varphi_1,\varphi_2,\ldots,\varphi_r\right\},
\end{equation}
and let us denote by $P^r:H_0^1(\Omega)\rightarrow \bU^r$ the $H_0^1$-orthogonal projection onto $\bU^r$. Then, it holds
\begin{equation}\label{cota_pod}
\frac{1}{N}\sum_{j=1}^N\|\nabla (I-P^r)y_h^j\|_0^2=\sum_{k={r+1}}^{d_r}\lambda_k.
\end{equation}
Going back to {the definition of $y_h^1,\ldots y_h^N$ in~\eqref{calU2},} from \eqref{cota_pod} we obtain
\begin{eqnarray}\label{cota_pod_b}
&&\sum_{l=0}^L\|\nabla(I-P^r)u_h^{\alpha_l}(t_0)\|_0^2
+\frac{1}{M+1}\sum_{j=1}^M\|\nabla(I-P^r)D^t u_h^{\alpha_0}(t_j^0)\|_0^2\nonumber\\
&&\quad +\frac{1}{N}\sum_{j=1}^M\sum_{l=1}^L\|\nabla(I-P^r)D^t D^\alpha u_h^{\alpha_l}(t_j^l)\|_0^2
=\sum_{k={r+1}}^{d_r}\lambda_k.
\end{eqnarray}
In Lemma \ref{lema2} below we prove that with the definition of the space ${\cal U}$ we can bound any of the terms $\|u_h^{\alpha_l}(t_j^l)-P^r u_h^{\alpha_l}(t_j^l)\|_X$ in terms of the quantity~$\Sigma_r$ regardless of the number of snapshots.
In many papers in the literature the authors need to assume that the errors in the average error bound \eqref{cota_pod} are equidistributed so that 
any of the terms $\|\nabla(y_h^j-P^r y_h^j)\|_0^2\approx\sum_{k={r+1}}^{d_r}\lambda_k$ but, as shown in \cite{rubino_etal}, this is not always true, see also \cite{Bosco_pointwise}. 

Before proving Lemma \ref{lema2} we need to prove the following auxiliary result.
\begin{lema} \label{lema1} Le $X$ be a Banach space and let $z(t_j,\alpha_l)\in X$, $0\le j\le M$, $0\le l\le L$, then
\begin{align}\label{inf_pri}
\|z(t_j^l,\alpha_l)\|_X^2\le& 3\|z(t_0,\alpha_l)\|_X^2+3(j\Delta t_l)\sum_{s=1}^j\Delta t_l\|D^t z(t_s^0,\alpha_0)\|_X^2 
\nonumber\\
&{}+3 (j\Delta t_l)(l\Delta \alpha)\sum_{s=1}^j\sum_{n=1}^l \Delta t_l\Delta\alpha\|D^\alpha D^t  z(t_s^n,\alpha_n)\|_X^2 
\end{align}
\end{lema}
\begin{proof}
We first observe that
\begin{eqnarray}\label{ladnu}
z(t_j^l,\alpha_l)=z(t_0^l,\alpha_l)+D^t z(t_1^l,\alpha_l)\Delta t_l+\ldots+D^t z(t_j^l,\alpha_l)\Delta t_l.
\end{eqnarray}
Now, for $1\le s \le j$, we add
$$
0 = \pm D^t z(t_s^{0},\alpha_{0}) \pm  D^t z(t_s^{1},\alpha_{1}) + \cdots +\pm D^t z(t_s^{l-1},\alpha_{l-1})
$$
to {each side of}  the identity $D^t z(t_s^{l},\alpha_{l}) = D^t z(t_s^{l},\alpha_{l})$ to get
\begin{align}
\label{ladtb}
D^t z(t_s^{l},\alpha_{l}) &= D^t z(t_s^{0},\alpha_{0}) + \sum_{n=1}^l \left(  D^t z(t_s^{n},\alpha_{n}) -  D^t z(t_s^{n-1},\alpha_{n-1})\right)
\nonumber\\
&{}=
 D^t z(t_s^{0},\alpha_{0}) + \sum_{n=1}^l \Delta\alpha D^\alpha D^t z(t_s^{n},\alpha_{n}).
\end{align}
Thus, from~\eqref{ladnu} we get
\begin{align*}
z(t_j^l,\alpha_l)=& z(t_0^l,\alpha_l) +\Delta t_l \sum_{s=1}^j D^t z(t_s^0,\alpha_0) + \Delta t_l \sum_{s=1}^j\sum_{n=1}^l \Delta \alpha D^\alpha D^t z(t_s^n, \alpha_n).
\end{align*}
And then, taking norms
\begin{eqnarray*}
\|z(t_j^l,\alpha_l)\|_X\le  \|z(t_0^l,\alpha_l)\|_X+\Delta t_l\sum_{s=1}^j\|D^t z(t_s^0,\alpha_0)\|_X+\Delta t_l\Delta \alpha\sum_{s=1}^j\sum_{n=1}^l \|D^\alpha D^t z(t_s^n,\alpha_n)\|_X .
\end{eqnarray*}
From which
\begin{align*}
\|z(t_j^l,\alpha_l)\|_X^2\le&  3\|z(t_0^l,\alpha_l)\|_X^2+3\left(\Delta t_l\sum_{s=1}^j\|D^t z(t_s^0,\alpha_0)\|_X\right)^2\\
&{}+3\left(\Delta t_l \Delta\alpha \sum_{s=1}^j\sum_{n=1}^l\|D^\alpha D^t z(t_s^n,\alpha_n)\|_X\right)^2.
\end{align*}
We finally reach \eqref{inf_pri} applying the discrete Cauchy-Schwarz inequality.
\end{proof}
\begin{lema} \label{lema2}
The following bound holds
\begin{eqnarray}\label{eq_point_2d}
\max_{0\le j\le M, 0\le l\le L}\|(I-P^r)u_h^{\alpha_l}(t_j^l)\|_X^2\le C_X \sum_{k={r+1}}^{d_r}\lambda_k,
\end{eqnarray}
for $C_X=C_{H_0^1}:=3\max\left(1,2T^2,4T^2(\alpha_L-\alpha_0)^2\right)$ if $X=H_0^1(\Omega)$ and
$C_X=C_{L^2}:=C_p^2 C_{H_0^1}$ if $X=L^2(\Omega)$ and $T=max_{0\le l\le L} T_l$.
\end{lema}
\begin{proof}
We take $z(t_j^l,\alpha_l)=(I-P^r)u_h^{\alpha_l}(t_j^l)$ and apply \eqref{inf_pri} from Lemma \ref{lema1}.
Then
\begin{align*}
\|(I-P^r)u_h^{\alpha_l}(t_j^l)\|_X^2&\le 3\|(I-P^r)u_h^{\alpha_l}(t_0)\|_X^2 +3(j\Delta t_l)\sum_{s=1}^j( \Delta t_l)\|(I-P^r)D^t u_h^{\alpha_0}(t_s^0)\|_X^2 
\nonumber\\
&\quad +3 (j\Delta t_l)(l\Delta \alpha)\sum_{s=1}^j\sum_{n=1}^l (\Delta\alpha)(\Delta t_l)\|(I-P^r)D^t D^\alpha  u_h^{\alpha_n}(t_s^n)\|_X^2.
\end{align*}
From which
\begin{align*}
\|(I-P^r)u_h^{\alpha_l}(t_j^l)\|_X^2&\le 3\|(I-P^r)u_h^{\alpha_l}(t_0)\|_X^2 +3 T\sum_{s=1}^M( \Delta t_l)\|(I-P^r)D^ t u_h^{\alpha_0}(t_s^0)\|_X^2 
\nonumber\\
&\quad +3 T (\alpha_L-\alpha_0)\sum_{s=1}^M\sum_{n=1}^L (\Delta\alpha)(\Delta t_l)\|(I-P^r)D^t D^\alpha u_h^{\alpha_n}(t_s^n)\|_X^2.
\end{align*}
Taking $L\ge 1$ and $M\ge 1$ so that $(L+1)/L\le 2$ and $(M+1)/M\le 2$ and denoting by
$C_{H_0^1}=3\max\left(1,2T^2,4T^2(\alpha_L-\alpha_0)^2\right)$ we get
\begin{align*}
\|(I-P^r)u_h^{\alpha_l}(t_j^l)\|_X^2&\le C_{H_0^1}\left(\|(I-P^r)u_h^{\alpha_l}(t_0)\|_X^2+\frac{1}{M+1}\sum_{s=1}^M\|(I-P^r)D^t u_h^{\alpha_0}(t_s^0)\|_X^2 \right.
 \\
 &\quad \left.+\frac{1}{N}\sum_{s=1}^M\sum_{n=1}^L \|(I-P^r)D^t D^\alpha u_h^{\alpha_n}(t_s^n)\|_X^2\right).
\end{align*}
Taking $X=H_0^1(\Omega)^d$ and applying \eqref{cota_pod_b} we get
\begin{equation}\label{here}
\|\nabla(I-P^r)u_h^{\alpha_l}(t_j^l)\|_0^2\le C_{H_0^1} \sum_{k={r+1}}^{d_r} \lambda_k.
\end{equation}
Since the above inequality is valid for any $0\le j\le M$ and $0\le l\le L$ taking the maximum we reach \eqref{eq_point_2d} for $X=H_0^1(\Omega)$. In the case $X=L^2(\Omega)$ we conclude applying Poincar\'e inequality \eqref{poincare} to \eqref{here} and taking the maximum.
\end{proof}
\section{Error analysis of the method}\label{sec4}
This section is devoted to analyze the method. The main result is Theorem \ref{main} in which we state the a priori bound for the error of  the method for any value of the parameter $\alpha \in [\alpha_0,\alpha_L]$ and any value of $t\in [0,T^\alpha]$. 

We will consider the following semi-discrete POD-ROM approximation to approach \eqref{eq:model}: Find $u_r^{\alpha}:(0,T^\alpha]\rightarrow \bU^r$
such that
\begin{equation}\label{eq:pod}
(u_{r,t}^{\alpha},v_r)+\nu_\alpha(\nabla u_r^{\alpha},\nabla v_r)+(g_\alpha(u_r^{\alpha}),v_r)=(f_\alpha,v_r),\quad \forall\ v_r\in \bU^r,
\end{equation}
with $u_r^{\alpha}(0)=u_{r,0}\in \bU^r$ and $u_{r,0}\approx u_0^\alpha$, $\alpha\in[\alpha_0,\alpha_L]$.
% We denote by $T=\max_{0\le l\le L}T_l$ and by $T_{\rm min}=\min_{0\le l\le L}T_l$ and assume $T_\alpha\le T_{\rm min}$.
For the error analysis we will assume that $g_\alpha$ is (globally) Lipschitz continuous since the general case can be
obtained arguing as in \cite[Theorem 2]{temporal_nos}. The proof of the following theorem can be found in \cite[Theorem 1]{temporal_nos}.
\begin{Theorem}\label{th1} Assume $g_\alpha$ is Lipschitz continuous with Lipschitz constant~$L_\alpha > 0$,
i.e., assume that $
|g_\alpha(s) - g_\alpha(t)| \le L_\alpha |s-t|$. 
Let $u_r^{\alpha}$ be the POD-ROM approximation solving \eqref{eq:pod} and let $P^r u_h^{\alpha}$ be the $H_0^1$-orthogonal projection 
onto $\bU^r$ of the semi-discrete Galerkin approximation $u_h^{\alpha}$ defined in \eqref{gal_semi}.
Then, for $K_\alpha=\frac{2}{T^\alpha} + 2L_\alpha$ and $K=\max_{[\alpha_0,\alpha_L]}K_\alpha$, the following bound holds for all~$t\in [0,T^\alpha]$
\begin{eqnarray}\label{er_pro_ur}
\lefteqn{\|u_r^{\alpha}(t)-P^r u_h^{\alpha}(t)\|_0^2+2\nu_\alpha\int_0^t \|\nabla (u_r^{\alpha}(s)-P^r u_h^{\alpha}(s))\|_0^2 \ \di s}\nonumber\\
 &\le& e^{Kt}\|u_r^{\alpha}(0)-P^r u_h^{\alpha}(0)\|_0\\
 &&+ e^{Kt}T_\alpha\left(\int_0^t \|(I-P^r)u^{\alpha}_{h,s}(s)\|_0^2\ \di s+L^2\int_0^t\|(I-P^r)u_h^{\alpha}(s)\|_0^2 \ \di s\right).\nonumber
\end{eqnarray}
\end{Theorem}
In view of \eqref{er_pro_ur}, to bound the error of the method (apart from the initial error) we need to bound the last two integral terms on the right-hand side of \eqref{er_pro_ur}. In next section we will bound those terms for the case in which $\alpha$ is one of the parameters used to compute the snapshots. These results will be used in Section \ref{sec42} to bound the integral terms in the case in which $\alpha$ is a parameter out of the sample. 
\subsection{The case $\alpha=\alpha_l$}
In this section we consider the case in which $\alpha=\alpha_l$, $l=0,\ldots,L$ and bound the terms:
$$
\int_0^t \|(I-P^r)u^{\alpha_l}_{h,s}(s)\|_0^2\ \di s,\quad \int_0^t\|(I-P^r)u_h^{\alpha_l}(s)\|_0^2 \ \di s,\ l=1,\ldots,L.
$$
To bound these terms we will apply the following auxiliary results that are proved in \cite[(26)]{temporal_nos} and
\cite[(36)]{temporal_nos}, respectively.
\begin{lema} Let $\varphi :[0,T]\times \Omega\rightarrow {\Bbb R}$  be a regular enough function. Then, the following bounds hold
\begin{equation}\label{new_cota_int_orp}
\int_{0}^T \|(I-P^r)\varphi (t)\|_0^2 \ \di t\le C\sum_{n=0}^M (\Delta t)\|(I-P^r) \varphi (t_n)\|_0^2+CC_p^2 (\Delta t)^{2q}
\int_{0}^{T}\left\| \frac{\partial^q \nabla \varphi (t)}{\partial t^q}\right\|_0^2 \ \di t,
\end{equation}
where $C$ depends only on~$q$ and~$c_{\rm int}$.
\begin{eqnarray}\label{deri_dif_p}
\sum_{n=0}^M(\Delta t )\|\varphi_t(t_n)\|_0^2\le C \sum_{n=1}^M (\Delta t)\|D^t \varphi(t_n)\|_0^2+C(\Delta t )^{2q} \int_{0}^T\left\|\frac{\partial^{q+1} \varphi(t)}{\partial t^{q+1}}\right\|_0^2 \ \di t,
\end{eqnarray}
 for $q\ge 2$ and $C$ depending only on~$q$.
 \end{lema}
\subsubsection{Bound for $
\int_0^t \|(I-P^r)u^{\alpha_l}_{h,s}(s)\|_0^2\ \di s$}
\begin{lema}\label{cota_t_nufix} For each $q\ge 2$, there exist a constant~$C$ such that, for $0\le l\le L$ and $1\le r\le {d_r}$, the following bound holds,
 \begin{equation}\label{eq:tra:pro_b}
\int_{0}^{T^l} \|(I-P^r)u_{h,t}^{\alpha_l}(t)\|_0^2\ \di t\le C C_p^2\left( C_0\sum_{k={r+1}}^{d_r}\lambda_k 
+(\Delta t)^{2q}
\int_{0}^{T^l}\left\| \frac{\partial^{q+1} \nabla u_h^{\alpha_l}(t)}{\partial t^{q+1}}\right\|_0^2 \ \di t\right),
\end{equation}
where $C_0=\max(4T,8(\alpha_L-\alpha_0)^2T)$, $T=\max_{0\le l\le L}T^l$, $\Delta t=\max_{0\le l\le L}\Delta t_l$, 
under the assumption that $u_h^{\alpha_l}$ is smooth enough so that the last 
term in \eqref{eq:tra:pro_b} is well defined.
\end{lema}
\begin{proof}
We apply \eqref{new_cota_int_orp}  to $\varphi(s)=u_{h,t}^{\alpha_l}(s)$,  $t_n=t_n^l$, $\Delta t=\Delta t_l$ and $T=T^l$ so that
\begin{eqnarray}\label{mecu1}
\int_{0}^{T^l} \|(I-P^r)u_{h,t}^{\alpha_l}(t)\|_0^2\ \di t&\le& C\sum_{n=0}^M (\Delta t_l)\|(I-P^r) u_{h,t}^{\alpha_l}(t_n^l)\|_0^2\nonumber\\
&&+CC_p^2 (\Delta t)^{2q}
\int_{0}^{T^l}\left\| \frac{\partial^{q+1} \nabla u_h^{\alpha_l}(t)}{\partial t^{q+1}}\right\|_0^2 \ \di t,
\end{eqnarray}
where in the last inequality we have applied that $\Delta t_l\le \Delta t$.
To bound the first term on the right-hand-side, 
we apply \eqref{deri_dif_p} to $\varphi=(I-P^r) u_{h}^{\alpha_l}$, use Poincar\'e's inequality \eqref{poincare}, 
and the fact that for any function $v\in H_0^1(\Omega)$ it holds $\|\nabla(I-P^r)  v\|_0\le \|\nabla v\|_0$ to obtain
\begin{eqnarray}
\lefteqn{
\sum_{n=0}^M (\Delta t_l)\|(I-P^r) u_{h,t}^{\alpha_l}(t_n^l)\|_0^2} \label{eq_pri}\\
&\le &  C \sum_{n=1}^M (\Delta t_l)\|D^t((I-P^r)u_h^{\alpha_l}(t_n^l))\|_0^2
 +C (\Delta t )^{2q} \int_{0}^{T^l}\left\|\frac{\partial^{q+1} (I-P^r)u_h^{\alpha_l}(t)}{\partial t^{q+1}}\right\|_0^2 \ \di t\nonumber\\
 &\le&  C C_p^2\sum_{n=1}^M (\Delta t_l)\|\nabla (I-P^r)D^tu_h^{\alpha_l}(t_n^l)\|_0^2
 +CC_p^2 (\Delta t )^{2q} \int_{0}^{T^l}\left\|\frac{\partial^{q+1} \nabla u_h^{\alpha_l}(t)}{\partial t^{q+1}}\right\|_0^2 \ \di t.\nonumber
\end{eqnarray}
To bound the first term on the right-hand side of \eqref{eq_pri} we observe that, arguing as in \eqref{ladtb},
we get
\begin{eqnarray*}
D^tu_h^{\alpha_l}(t_n^l)=D^tu_h^{\alpha_0}(t_n^0)+D^\alpha D^tu_h^{\alpha_1}(t_n^1)\Delta \alpha+\ldots+D^\alpha D^t u_h^{\alpha_l}(t_n^l)\Delta \alpha.
\end{eqnarray*}
Then,
\begin{eqnarray*}
\|\nabla(I-P^r)D^tu_h^{\alpha_l}(t_n^l)\|_0\le \|\nabla(I-P^r)D^tu_h^{\alpha_0}(t_n^0)\|_0+\sum_{j=1}^l\|\nabla(I-P^r)D^\alpha D^t u_h^{\alpha_j}(t_n^j)\|_0\Delta \alpha.
\end{eqnarray*}
So that
\begin{eqnarray*}
&&\|\nabla(I-P^r)D^tu_h^{\alpha_l}(t_n^l)\|_0^2\le 2\|\nabla(I-P^r)D^tu_h^{\alpha_0}(t_n^0)\|_0^2\nonumber\\
 &&\quad+2\left(\sum_{j=1}^l\|\nabla(I-P^r)D^\alpha D^t u_h^{\alpha_j}(t_n^j)\|_0\Delta \alpha\right)^2\\
&&\quad\le 2\|\nabla(I-P^r)D^tu_h^{\alpha_0}(t_n^0)\|_0^2+2(l\Delta \alpha)
\sum_{j=1}^l\|\nabla(I-P^r)D^\alpha D^t u_h^{\alpha_j}(t_n^j)\|_0^2\Delta \alpha.
\end{eqnarray*}
As a consequence,
\begin{eqnarray*}
&&\sum_{n=1}^M (\Delta t_l)\|\nabla (I-P^r)D^tu_h^{\alpha_l}(t_n)\|_0^2\le 2\sum_{n=1}^M (\Delta t_l)\|\nabla(I-P^r)D^tu_h^{\alpha_0}(t_n^0)\|_0^2\nonumber\\
&&\quad+2(\alpha_L-\alpha_0)\sum_{n=1}^M \sum_{j=1}^L\|\nabla(I-P^r)D^\alpha D^t u_h^{\alpha_j}(t_n^j)\|_0^2(\Delta t_l)(\Delta \alpha)
\nonumber\\
&&\quad\le \frac{4T}{M+1}\sum_{n=1}^M \|\nabla(I-P^r)D^tu_h^{\alpha_0}(t_n^0)\|_0^2
\nonumber\\
&&\quad \ +\frac{8(\alpha_L-\alpha_0)^2T}{N}\sum_{n=1}^M \sum_{j=1}^L\|\nabla(I-P^r)D^\alpha D^t u_h^{\alpha_j}(t_n^j)\|_0^2\nonumber\\
&&\le C_0\left(\frac{1}{M+1}\sum_{n=1}^M \|\nabla(I-P^r)D^tu_h^{\alpha_0}(t_n^0)\|_0^2
+\frac{1}{N}\sum_{n=1}^M \sum_{j=1}^L\|\nabla(I-P^r)D^\alpha D^t u_h^{\alpha_j}(t_n^j)\|_0^2\right),
\end{eqnarray*}
where $C_0=\max(4T,8(\alpha_L-\alpha_0)^2T)$ and we have bounded $(M+1)/M\le 2$, $(L+1)/L\le 2$.
Applying \eqref{cota_pod_b} and taking into account that $D^\alpha D^t=D^t D^\alpha$ we finally obtain
\begin{eqnarray}\label{eq:seg}
\sum_{n=1}^M (\Delta t_l)\|\nabla (I-P^r)D^tu_h^{\alpha_l}(t_n^l)\|_0^2\le C_0\sum_{k=r+1}^{d_r}\lambda_k.
\end{eqnarray}
Inserting \eqref{eq:seg} into \eqref{eq_pri} and \eqref{eq_pri} into \eqref{mecu1} we  reach \eqref{eq:tra:pro_b}.
\end{proof}

\subsubsection{Bound for $
\int_0^t \|(I-P^r)u^{\alpha_l}_{h}(s)\|_0^2\ \di s$}

\begin{lema}\label{cota_nufixb} For each $q\ge 2$,  there exist a constant~$C$ such that, for $0\le l\le L$ and $1\le r\le {d_r}$, the following bound holds,
 \begin{equation}\label{eq:tra:pro_c}
\int_{0}^{T^l} \|(I-P^r)u_{h}^{\alpha_l}(t)\|_0^2\ \di t\le C \left( TC_X\sum_{k={r+1}}^{d_r}\lambda_k 
+C_p^2(\Delta t)^{2q}
\int_{0}^{T^l}\left\| \frac{\partial^{q} \nabla u_h^{\alpha_l}(t)}{\partial t^{q}}\right\|_0^2 \ \di t\right),
\end{equation} 
where $C_X$ is the constant in \eqref{eq_point_2d}, $T=\max_{0\le l\le L}T_l$, $\Delta t=\max_{0\le l\le L}\Delta t$ and
under the assumption that $u_h^{\alpha_l}$ is smooth enough so that that the last 
term in \eqref{eq:tra:pro_c} is well defined.
\end{lema}
\begin{proof}
We apply \eqref{new_cota_int_orp}  to $\varphi(s)=u_{h}^{\alpha_l}(s)$, $t_n=t_n^l$, $\Delta t=\Delta t_l$ and $T=T_l$ so that
\begin{eqnarray}\label{mecu1b}
\int_{0}^{T^l} \|(I-P^r)u_{h}^{\alpha_l}(t)\|_0^2\ \di t&\le& C\sum_{n=0}^M (\Delta t_l)\|(I-P^r) u_{h}^{\alpha_l}(t_n^l)\|_0^2\nonumber\\
&&+CC_p^2 (\Delta t)^{2q}
\int_{0}^{T^l}\left\| \frac{\partial^{q} \nabla u_h^{\alpha_l}(t)}{\partial t^{q}}\right\|_0^2 \ \di t,
\end{eqnarray}
where we have used $\Delta t_l\le \Delta t$.
To bound the first term on the right-hand side of \eqref{mecu1b} we apply \eqref{eq_point_2d}.
Then
\begin{eqnarray*}
\sum_{n=0}^M (\Delta t_l)\|(I-P^r) u_{h}^{\alpha_l}(t_n^l)\|_0^2\le (M+1)\Delta t_l C_X \sum_{k={r+1}}^{d_r}\lambda_k
\le 2T C_X \sum_{k={r+1}}^{d_r}\lambda_k,
\end{eqnarray*}
that inserted into \eqref{mecu1b} gives \eqref{eq:tra:pro_c}.
\end{proof}
\subsection{The case $\alpha\neq\alpha_l$}\label{sec42}
As stated before, in this section we will bound the terms:
$$
\int_0^t \|(I-P^r)u^{\alpha}_{h,s}(s)\|_0^2\ \di s,\quad \int_0^t\|(I-P^r)u_h^{\alpha}(s)\|_0^2 \ \di s,\ \alpha\in[\alpha_0,\alpha_L].
$$
We will apply the bounds in Lemmas \ref{cota_t_nufix} and \ref{cota_nufixb} above. To this
end, we need a previous auxiliary result. Let us denote
$$
I = \bigcup_{\alpha\in[\alpha_0,\alpha_{L}]} \{\alpha\} \times [0,T^\alpha].
$$
\begin{lema}\label{internu}
Let $v:I\times{\Omega}\rightarrow {\Bbb R}$ be a smooth enough
function and let us denote $v^\alpha(s)=v(\alpha,s,\cdot)$, for $s\in[0,T^\alpha]$.
Fix $2\le m\le L$. Then, there exist a constant~$C=C(m)$ such that for $l$ satisfying
$\alpha\in [\alpha_l,\alpha_{l+m-1}]\subset [\alpha_0,\alpha_L]$,
the following bounds hold,
\begin{align}\label{nueva_inter}
%\int_0^{T_{\alpha}} \|w^\alpha(s)\|_0^2\ \di s\le C\sum_{l=0}^L\int_0^{T_\alpha} \|w^{\alpha_l}(s)\|_0^2\ \di s+C (\Delta \alpha)^4\int_0^{T_{\alpha}} \max_{\alpha\in[\alpha_0,\alpha_L]}\left\|\frac{\partial^2w^{\alpha}(s)}{\partial\alpha^2}\right\|_0^2\ \di s.
\int_0^{T^\alpha} \|v^\alpha(s)\|_0^2\ \di s\le& C\sum_{j=l}^{l+m-1}\int_0^{T^{j}} \|v^{\alpha_j}(t)\|_0^2\ \di t
\nonumber\\
&{}
+C (\Delta \alpha)^{2m-1}\int_{\alpha_l}^{\alpha_{l+m-1}}\int_0^{T^\mu} 
\sum_{i+j\le m}\left\|\frac{\partial^{i+j}v^{\mu}(t)}{\partial t^i\partial\mu^{j}}\right\|_0^2\,\di t\,\di\mu,
\end{align}
\begin{align}\label{nueva_inter3}
\|v^\alpha(s)\|_0^2\ \le& C\sum_{j=l}^{l+m-1} \|v^{\alpha_j}(T^j s/T^\alpha)\|_0^2\nonumber\\
&{}+C (\Delta \alpha)^{2m-1}
\int_{\alpha_l}^{\alpha_{l+m-1}}\sum_{i+j\le m}\left\|\frac{\partial^{i+j}v^{\mu}(T^\mu s/T^\alpha)}{\partial t^i\partial \mu^{j}}\right\|_0^2 \di \mu,\quad s\in[0,T^\alpha],
\end{align}
\end{lema}
for $T^j=T^{\alpha_j}$.
\begin{proof}
We consider $w:[\alpha_l,\alpha_{l+m-1}]\times[0,T^{\alpha}] \times\Omega \rightarrow {\Bbb R}$ given by
\begin{equation}
\label{law}
w(\mu,s,x)=v(\mu,(T^{\mu}/T^{\alpha})s,x),
\end{equation}
and denote~$w^{\mu}(s)=w(\mu,s,\cdot)$.
We first treat the case $m=2$.
Using the linear interpolant (with respect to the parameter), we can write
$$
w^\alpha=I_{2,\alpha}w^{\alpha} +R_{2,\alpha}w^{\alpha},\quad I_{2,\alpha}w^{\alpha} =w^{\alpha_l}\frac{(\alpha_{l+1}-\alpha)}{\Delta \alpha}
+w^{\alpha_{l+1}}\frac{(\alpha-\alpha_l)}{\Delta \alpha}.
$$
For any $s\in[0,T^\alpha]$ 
$$
\|w^{\alpha}(s)\|_0\le \|w^{\alpha_1}(s)\|_0+\|w^{\alpha_2}(s)\|_0+\|R_{2,\alpha}w^{\alpha}(s)\|_0.
$$
And then, 
$$
\int_0^{T^\alpha}\|w^{\alpha}(s)\|_0^2\ \di s\le 3\int_0^{T^\alpha}\|w^{\alpha_l}(s)\|_0^2\ \di s+3\int_0^{T^{\alpha}}\|w^{\alpha_{l+1}}(s)\|_0^2\ \di t+3\int_0^{T^\alpha}\|R_{2,\alpha}w^{\alpha}(s)\|_0^2\ \di s,
$$
so that, with the change of variables $s=T^\alpha t/T^{\mu}$, for $\mu=\alpha_l$ and $\mu=\alpha_{l+1}$ in the first and second integrals on the right-hand side, respectively, we get
$$
\int_0^{T^\alpha}\|v^{\alpha}(s)\|_0^2\ \di s\le C\int_0^{T^{l}}\|v^{\alpha_l}(t)\|_0^2\ \di t+C\int_0^{T^{{l+1}}}\|v^{\alpha_{l+1}}(t)\|_0^2\ \di t+3\int_0^{T^\alpha}\|R_{2,\alpha}w^{\alpha}(s)\|_0^2\ \di s.
$$
To bound the last term above, we observe that, applying \eqref{cota_inter_inf} to the Lagrange interpolant with respect to
$\alpha$ with $n=2$, we get for any $s\in[0,T^\alpha]$ and $x\in \Omega$.
$$
|R_{2,\alpha}w(\alpha,s,x)|\le c_\text{\rm int} (\Delta \alpha)^{2-1/2}\biggl( \int_{\alpha_l}^{\alpha_{l+1}} \left|\frac{\partial^2w(\mu,s,x)}{\partial\mu^2}\right|\,\di \mu\biggr)^{1/2}.
$$
%The interpolation constant $c_\text{\rm int}$ could depend on $s$ and $x$ (acting as parameters) so that we
%take the maximum constant for $s\in[0,T]$, $x\in \Omega$. 
Then,
$$
\|R_{2,\alpha}w^{\alpha}(s)\|_0^2\le c_\text{\rm int}^2 (\Delta \alpha)^{3}\ \int_{\alpha_l}^{\alpha_{l+1}} \left\|\frac{\partial^2w^{\mu}(s)}{\partial\mu^2}\right\|_0^2\,\di \mu,
$$
so that
\begin{align*}
\int_0^{T^\alpha}\|R_{2,\alpha}w^{\alpha}(s)\|_0^2\ \di s& \le c_\text{\rm int}^2 (\Delta \alpha)^3\int_0^{T^\alpha} \int_{\alpha_l}^{\alpha_{l+1}} \left\|\frac{\partial^2w^{\mu}(s)}{\partial\mu^2}\right\|_0^2\,\di \mu\, \di s,
\nonumber\\
 &= c_\text{\rm int}^2 (\Delta \alpha)^3 \int_{\alpha_l}^{\alpha_{l+1}} \int_0^{T^\alpha}\left\|\frac{\partial^2w^{\mu}(s)}{\partial\mu^2}\right\|_0^2\, \di s\,\di \mu,
\end{align*}
from which we reach \eqref{nueva_inter} for $m=2$ via the change of variables $s=T^\alpha t/T^{\mu} $ and~\eqref{law}. The case $m>2$ can be proved with a similar argument. The proof of \eqref{nueva_inter3} follows also easily with the same arguments.
\end{proof}
\begin{remark}\label{re1}
We can replaced in \eqref{nueva_inter}  $(\Delta \alpha)^{2m-1}$  by $(\Delta \alpha)^{2m}$ if we use in the proof of Lemma 5
the error bound \eqref{cota_inter} instead of \eqref{cota_inter_inf}. In that case, stronger regularity is required and the full term
$$C (\Delta \alpha)^{2m-1}\int_{\alpha_l}^{\alpha_{l+m-1}}\int_0^{T^\mu} 
\sum_{i+j\le m}\left\|\frac{\partial^{i+j}v^{\mu}(t)}{\partial t^i\partial\mu^{j}}\right\|_0^2\,\di t\,\di\mu,$$
in \eqref{nueva_inter} should be replaced by
$$C (\Delta \alpha)^{2m}\int_0^{T^\mu} \max_{\mu\in[{\alpha_l},{\alpha_{l+m-1}}]}
\sum_{i+j\le m}\left\|\frac{\partial^{i+j}v^{\mu}(t)}{\partial t^i\partial\mu^{j}}\right\|_0^2\,\di t.$$
Analogously, we can replace in \eqref{nueva_inter3}
$$C (\Delta \alpha)^{2m-1}
\int_{\alpha_l}^{\alpha_{l+m-1}}\sum_{i+j\le m}\left\|\frac{\partial^{i+j}v^{\mu}(T^\mu s/T^\alpha)}{\partial t^i\partial \mu^{j}}\right\|_0^2 \di \mu,\quad s\in[0,T^\alpha],$$
by
$$C (\Delta \alpha)^{2m}
\max_{\mu\in[{\alpha_l},{\alpha_{l+m-1}}]}\sum_{i+j\le m}\left\|\frac{\partial^{i+j}v^{\mu}(T^\mu s/T^\alpha)}{\partial t^i\partial \mu^{j}}\right\|_0^2,\quad s\in[0,T^\alpha].$$
\end{remark}
\subsubsection{Bound for $
\int_0^t \|(I-P^r)u^{\alpha}_{h,s}(s)\|_0^2\ \di s$}
\begin{lema}\label{lema_nu1}
For each $q\ge 2$, $m\ge 2$ there exist a constant~$C$ such that for $1\le r\le {d_r}$ the following bound holds
\begin{eqnarray}\label{eq:tra:pro_b_nu}
&&\int_{0}^{T^\alpha} \|(I-P^r)u_{h,t}^{\alpha}(t)\|_0^2\ \di t\le mC C_p^2 C_0\sum_{k={r+1}}^{d_r}\lambda_k \nonumber\\
&&\quad+C C_p^2(\Delta t)^{2q}\sum_{j=l}^{l+m-1}
\int_{0}^{T^j}\left\| \frac{\partial^{q+1} \nabla u_h^{\alpha_j}(t)}{\partial t^{q+1}}\right\|_0^2 \ \di t
\nonumber\\
&&\quad+ C C_p^2(\Delta \alpha)^{2m-1}\int_{\alpha_l}^{\alpha_{l+m-1}}\int_0^{T^\mu} 
\sum_{i+j\le m}\left\|\frac{\partial^{i+j+1}\nabla u_{h}^{\mu}(t)}{\partial t^{i+1}\partial\mu^{j}}\right\|_0^2\ \di t \di \mu,\nonumber
\end{eqnarray}
where $C_0=\max(4T,8(\alpha_L-\alpha_0)^2T)$, $T=\max_{0\le l\le L}T^l$, $\Delta t=\max_{0\le l\le L}\Delta t$,
under the assumption that $u_h^{\alpha_l}$, $u_h^{\alpha}$ are smooth enough such that the last 
two terms in \eqref{eq:tra:pro_b_nu} are well defined.
\end{lema}
\begin{proof}
We first apply \eqref{nueva_inter} to $v^\alpha=(I-P^r)u_{h,t}^{\alpha}$ to get 
\begin{eqnarray*}
&&\int_0^{T^\alpha} \|(I-P^r)u_{h,s}^{\alpha}(s)\|_0^2\ \di s\le C\sum_{j=l}^{l+m-1}\int_0^{T^j} \|(I-P^r)u_{h,s}^{\alpha_j}(s)\|_0^2\ \di s
\nonumber\\
&&\quad+C(\Delta \alpha)^{2m-1}C_p^2 \int_{\alpha_l}^{\alpha_{l+m-1}}\int_0^{T^\mu} 
\sum_{i+j\le m}\left\|\frac{\partial^{i+1+j}\nabla u_{h}^{\mu}(t)}{\partial t^{i+1}\partial\mu^{j}}\right\|_0^2\ \di t \di \mu.
\end{eqnarray*}
To bound the second term we apply Poincar\'e inequality \eqref{poincare} taking into account $\|\nabla (I-P^r) v\|_0\le \|\nabla v\|_0$,
for any function $v	\in H_0^1(\Omega)$. Then, we get
\begin{eqnarray*}
&&\int_{\alpha_l}^{\alpha_{l+m-1}}\int_0^{T^\mu} 
\sum_{i+j\le m}\left\|\frac{\partial^{i+j}(I-P^r)u_{h,t}^{\mu}(t)}{\partial t^i\partial\mu^{j}}\right\|_0^2\ \di t \di \mu
\nonumber\\
&&\quad \le C_p^2 \int_{\alpha_l}^{\alpha_{l+m-1}}\int_0^{T^\mu} 
\sum_{i+j\le m}\left\|\frac{\partial^{i+j}\nabla u_{h,t}^{\mu}(t)}{\partial t^i\partial\mu^{j}}\right\|_0^2\ \di t \di \mu
\end{eqnarray*}
To bound the first term we apply Lemma \ref{cota_t_nufix} for $l=0,\ldots,L$. Then
\begin{eqnarray*}
\sum_{j=l}^{l+m-1}\int_0^{T^j} \|(I-P^r)u_{h,s}^{\alpha_j}(s)\|_0^2\ \di s
 &\le&m C C_p^2 C_0\sum_{k={r+1}}^{d_r}\lambda_k 
\nonumber\\&&\quad+C C_p^2(\Delta t)^{2q}\sum_{j=l}^{l+m-1}
\int_{0}^{T^j}\left\| \frac{\partial^{q+1} \nabla u_h^{\alpha_j}(t)}{\partial t^{q+1}}\right\|_0^2 \ \di t,
\end{eqnarray*}
which finishes the proof.
\end{proof}
\subsubsection{Bound for $
\int_0^t \|(I-P^r)u^{\alpha}_{h}(s)\|_0^2\ \di s$}
\begin{lema}\label{lema_nu2}
For each $q\ge 2$, $m\ge 2$, there exist a constant~$C$ such that for $1\le r\le {d_r}$ the following bound holds
\begin{eqnarray}\label{eq:tra:pro_b_nu_sin}
&&\int_{0}^{T^\alpha} \|(I-P^r)u_{h}^{\alpha}(t)\|_0^2\ \di t\le mC TC_X\sum_{k={r+1}}^{d_r}\lambda_k \nonumber
\\
&&\quad+C C_p^2(\Delta t)^{2q}\sum_{j=l}^{l+m-1}
\int_{0}^{T^j}\left\| \frac{\partial^{q} \nabla u_h^{\alpha_j}(t)}{\partial t^{q}}\right\|_0^2 \ \di t\nonumber\\
&&\quad+ C C_p^2(\Delta \alpha)^{2m-1}\int_{\alpha_l}^{\alpha_{l+m-1}}\int_0^{T^\mu} 
\sum_{i+j\le m}\left\|\frac{\partial^{i+j}\nabla u_{h}^{\mu}(t)}{\partial t^i\partial\mu^{j}}\right\|_0^2\ \di t \di \mu,
\end{eqnarray}
where $C_0=\max(4T,8(\alpha_L-\alpha_0)^2T)$, $T=\max_{0\le l\le L}T^l$, $\Delta t=\max_{0\le l\le L}\Delta t$,
under the assumption that $u_h^{\alpha_l}$, $u_h^{\alpha}$ are smooth enough such that the last 
two terms in \eqref{eq:tra:pro_b_nu_sin} are well defined.
\end{lema}
\begin{proof}
The proof is the same as the proof of Lemma \ref{lema_nu1} but applying \eqref{nueva_inter} to $v^\alpha=(I-P^r)u_{h}^{\alpha}$ instead of $v^\alpha=(I-P^r)u_{h,t}^{\alpha}$ and applying Lemma \ref{cota_nufixb}  instead of Lemma \ref{cota_t_nufix} for $l=0,\ldots,L$.
\end{proof}
Inserting \eqref{eq:tra:pro_b_nu} and \eqref{eq:tra:pro_b_nu_sin} into \eqref{er_pro_ur} we reach for all
$t\in[0,T^\alpha]$ and $\alpha\in[\alpha_0,\alpha_L]$
\begin{eqnarray}\label{cota_fin1}
\lefteqn{\|u_r^{\alpha}(t)-P^ru_h^\alpha(t)\|_0^2+2\nu_\alpha\int_0^t \|\nabla (u_r^{\alpha}(s)-P^ru_h^\alpha(s))\|_0^2 \ \di s}\nonumber\\
 &\le& e^{Kt}\|u_r^{\alpha}(0)-P^ru_h^\alpha(0)\|_0+C e^{Kt}T\left(m C_0+TC_X\right)\sum_{k={r+1}}^{d_r}\lambda_k\nonumber\\
 &&+C e^{Kt}TC_p^2\left((\Delta t)^{2q}\sum_{j=l}^{l+m-1}
\int_{0}^{T^j}\left(\left\| \frac{\partial^{q} \nabla u_h^{\alpha_j}(t)}{\partial t^{q}}\right\|_0^2 +
\left\| \frac{\partial^{q+1} \nabla u_h^{\alpha_j}(t)}{\partial t^{q+1}}\right\|_0^2\right)\ \di t\right.\\
&&\left.+(\Delta \alpha)^{2m-2/m}\int_{\alpha_l}^{\alpha_{l+m-1}}\int_0^{T_\mu} 
\sum_{i+j\le m}\left(\left\|\frac{\partial^{i+j+1}\nabla u_{h}^{\mu}(t)}{\partial t^{i+1}\partial\mu^{j}}\right\|_0^2+\left\|\frac{\partial^{i+j}\nabla u_{h}^{\mu}(t)}{\partial t^i\partial\mu^{j}}\right\|_0^2\ \di t \di \mu \right)\right).\nonumber
\end{eqnarray}
Now, we observe that we can bound
\begin{equation}\label{decom}
\|u_r^{\alpha}(t)-u_h^{\alpha}(t)\|_0^2\le 2\|u_r^{\alpha}(t)-P^ru_h^\alpha(t)\|_0^2+2\|(I-P^r) u_h^{\alpha}(t)\|_0^2,
\end{equation}
and apply \eqref{cota_fin1} to bound the first term on the right-hand side above. To bound the second term, in the case $\alpha=\alpha_l$, $l=0,\ldots, L$, we apply \cite[(67)]{temporal_nos} 
\begin{eqnarray*}
\|(I-P^r)\varphi(t)\|_0^2 
&\le& C q \max_{0\le n\le M}\|(I-P^r) \varphi(t_n)\|_0^2+C (\Delta t)^{2q}
\max_{0\le t\le T}\left\| \frac{\partial^q \nabla \varphi(t)}{\partial t^{q}}\right\|_0^2,
\end{eqnarray*} 
to $u_h^{\alpha_l}$ with $t_n=t_n^l$, $\Delta t =\Delta t_l$ and taking into account that $\Delta t_l\le \Delta t$, to get
\begin{eqnarray*}
\|(I-P^r)u_h^{\alpha_l}(t)\|_0^2 
&\le& C q \max_{0\le n\le M}\|(I-P^r)u_h^{\alpha_l}(t_n^l)\|_0^2+C (\Delta t)^{2q}
\max_{0\le t\le T^l}\left\| \frac{\partial^q \nabla u_h^{\alpha_l}(t)}{\partial t^{q}}\right\|_0^2.
\end{eqnarray*} 
And, applying \eqref{eq_point_2d}, for $t\in [0,T^l]$, 
\begin{eqnarray}\label{vale}
\|(I-P^r)u_h^{\alpha_l}(t)\|_0^2 
&\le& q C  C_{L^2} \sum_{k={r+1}}^{d_r}\lambda_k+C (\Delta t)^{2q}
\max_{0\le t\le T^l}\left\| \frac{\partial^q \nabla u_h^{\alpha_l}(t)}{\partial t^{q}}\right\|_0^2.
\end{eqnarray}
If $\alpha\neq\alpha_l$  we apply \eqref{nueva_inter3} to $v^\alpha=(I-P^r)u_h^{\alpha}$ and argue as before to reach for $t\in[0, T^\alpha]$
\begin{eqnarray}\label{nueva_inter4}
\|(I-P^r)u_h^{\alpha}(t)\|_0^2\ &\le& C\sum_{j=l}^{l+m-1} \|(I-P^r)u_h^{\alpha_j}(T^j t/T_\alpha)\|_0^2\\
&&\quad+C C_p^2(\Delta \alpha)^{2m-2/m}
\int_{\alpha_l}^{\alpha_{l+m-1}}\sum_{i+j\le m}\left\|\frac{\partial^{i+j}\nabla u_h^{\mu}(T^\mu t/T^\alpha)}{\partial t^i\partial \mu^{j}}\right\|_0^2 \di \mu.\nonumber
%\\
%&\le& L C\max_{0\le l\le L} \|(I-P^r)u_h^{\alpha_l}(t)\|_0^2\ \di s+C (\Delta \alpha)^{2m}
%\max_{\alpha\in[\alpha_0,\alpha_L]}\left\|\frac{\partial^{m}\nabla u_h^{\alpha}(t)}{\partial\alpha^{m}}\right\|_0^2\nonumber.
\end{eqnarray}
Inserting \eqref{vale} into \eqref{nueva_inter4}
we conclude for $t\in[0,T^\alpha]$
\begin{eqnarray}\label{nueva_inter5}
\|(I-P^r)u_h^{\alpha}(t)\|_0^2
&\le& q m C  C_{L^2} \sum_{k={r+1}}^{d_r}\lambda_k+C  (\Delta t)^{2q}
\sum_{j=l}^{l+m-1}\max_{0\le t\le T^j}\left\| \frac{\partial^q \nabla u_h^{\alpha_j}(t)}{\partial t^{q}}\right\|_0^2\\
&&\quad+C C_p^2(\Delta \alpha)^{2m-2/m}
\int_{\alpha_l}^{\alpha_{l+m-1}}\sum_{i+j\le m}\left\|\frac{\partial^{i+j}\nabla u_h^{\mu}(T^\mu t/T^\alpha)}{\partial t^i\partial \mu^{j}}\right\|_0^2 \di \mu.\nonumber
\end{eqnarray}
With \eqref{decom}, \eqref{cota_fin1} and \eqref{nueva_inter5} we conclude the following theorem.
\begin{Theorem}\label{main}
Assume $u_r^{\alpha}(t)=P^ru_h(0)$. For all $t\in[0,T^\alpha]$ and $\alpha\in[\alpha_0,\alpha_L]$. Then, there exists a constant $C$
depending on $T$, $\alpha_L-\alpha_0$ and $C_p$  such that the following bound holds for $q\ge2$, $m\ge2,$ whenever the functions $u_h^{\alpha}$ are smooth enough so that all the terms are well defined 
\begin{eqnarray}\label{max_puntual_todo}
\|u_r^{\alpha}(t)-u_h^{\alpha}(t)\|_0^2\le C_{m,q}
\sum_{k={r+1}}^{d_r}\lambda_k+C_{1,u_h}(\Delta t)^{2q}+C_{2,u_h}(\Delta \alpha)^{2m-1},
\end{eqnarray}
where
\begin{align*}
C_{m,q}&=C((m+1)+qm),\\
C_{1,u_h}&=C\sum_{j=l}^{l+m-1}\max_{0\le t\le {T^j}}\left\| \frac{\partial^q \nabla u_h^{\alpha_j}(t)}{\partial t^{q}}\right\|_0^2+
C\int_{0}^{T^j}\left(\left\| \frac{\partial^{q} \nabla u_h^{\alpha_j}(t)}{\partial t^{q}}\right\|_0^2 +
\left\| \frac{\partial^{q+1} \nabla u_h^{\alpha_j}(t)}{\partial t^{q+1}}\right\|_0^2\right)\ \di t,\\
C_{2,u_h}&=C\int_{\alpha_l}^{\alpha_{l+m-1}}\int_0^{T_\mu},
\sum_{i+j\le m}\left(\left\|\frac{\partial^{i+j+1}\nabla u_{h}^{\mu}(t)}{\partial t^{i+1}\partial\mu^{j}}\right\|_0^2+\left\|\frac{\partial^{i+j}\nabla u_{h}^{\mu}(t)}{\partial t^i\partial\mu^{j}}\right\|_0^2\ \di t \di \mu \right)\\
&\quad+C\int_{\alpha_l}^{\alpha_{l+m-1}}\sum_{i+j\le m}\left\|\frac{\partial^{i+j}\nabla u_h^{\mu}(T^\mu t/T^\alpha)}{\partial t^i\partial \mu^{j}}\right\|_0^2 \di \mu.
\end{align*}  
\end{Theorem}
\begin{remark}
As explained in Remark \ref{re1}, in the last term of \eqref{max_puntual_todo} we can replaced $(\Delta \alpha)^{2m-1}$ by
$(\Delta \alpha)^{2m}$ with stronger regularity assumptions.

In view of the error bound \eqref{max_puntual_todo} we observe that the error depends on the quantity~$\Sigma_r$, on the
distance between two consecutive values in time, $\Delta t$, and on the distance between two consecutive
values of the parameter, $\Delta \alpha$. We observe that the rate of convergence in terms of $\Delta t$ and $\Delta \alpha$ depends on the regularity (with respect to~$t$ and~$\alpha$) of the finite-element approximations, while the constants multiplying the powers of~$\Delta t$ and~$\Delta \alpha$ depend on the size of the derivatives (with respect to~$t$ and~$\alpha$) of these finite-element approximations. Let us finally observe that the bound holds for different values of $q$ and $m$ and there is no way to know beforehand which values of $q$ and $m$ will give the smallest error in a given practical computation, although some estimations could be done by computing derivatives of the finite element approximations (on which the basis functions of the POD method are based).
\end{remark}
\section{The standard method}\label{sec5}
There are many papers in the literature in which instead of the method we propose they consider the dataset
consisting on the snapshots $u_h^{\alpha_l}(t_j^l)$, for $l=0,\ldots,L$ and $j=0,\ldots,M$. For simplicity let us assume that all the functions are defined in the same time interval $[0,T]$ so that the times $t_j^l=t_j$ do
not depend on $l$. If this is not the case, one can make a change of variable, arguing as in the proof of Lemma \ref{internu},
to shift all functions to the chosen interval $[0,T]$.
In the standard case, the set of snapshots is
$$
\bU= {\rm span}\left\{u_h^{\alpha_l}(t_j)\right\},\quad l=0,\ldots,L, \quad j=0,\ldots,M.
$$
Then, instead of \eqref{cota_pod_b}, one has
\begin{eqnarray}\label{39}
\frac{1}{N}\sum_{j=0}^M\sum_{l=0}^L\|\nabla (I-P^r)u_h^{\alpha_l}(t_j)\|_0^2=\sum_{k={r+1}}^{\tilde d_r}\tilde \lambda_k,
\end{eqnarray}
where $\tilde \lambda_k$ are the corresponding eigenvalues in the singular value decomposition. Let us observe that in this case we cannot prove \eqref{eq_point_2d} (for $\tilde \lambda_k$) although many authors assume in the error analysis that \eqref{eq_point_2d} holds. However, as stated in \cite{rubino_etal}, this is not always true. Following
\cite{Bosco_pointwise}, we can prove a priori bounds for the standard method without making such assumption. The argument reads as follows.

For any $l=0,\ldots,L$ we denote by 
$$
\overline\sigma^{\alpha_l}_{u_h}=\frac{1}{M+1}\sum_{j=0}^M\|\nabla (I-P^r)u_h^{\alpha_l}(t_j)\|_0^2,
$$
the average error.
Following the error analysis in \cite[Theorem 2.2]{Bosco_pointwise} one can prove that for any $l=0,\ldots,L$, there exists a constant $C$ such 
that the following bound holds
\begin{eqnarray*}\label{poin1}
\max_{0\le j\le M} 
\bigl\|\nabla (I-P^r)u_h^{\alpha_l}(t_j)\|_0^2 &\le& C \left(\overline\sigma^{\alpha_l}_{u_h}\right)^{{1}-\frac{1}{2m}}\bigl\| (I-P^r)\partial_t u_h^{\alpha_l}\bigr\|_{H^{m-1}(0,T,H_0^1)}^{\frac{1}{m}}
+4\overline\sigma^{\alpha_l}_{u_h}.
\end{eqnarray*}
Since the second term on the right-hand side above decays faster than the first one, we will write in the sequel the following inequality (that holds for a different constant $C$)
\begin{eqnarray}\label{poin1}
\max_{0\le j\le M} 
\bigl\|\nabla (I-P^r)u_h^{\alpha_l}(t_j)\|_0^2 &\le& C \left(\overline\sigma^{\alpha_l}_{u_h}\right)^{{1}-\frac{1}{2m}}\bigl\| (I-P^r)\partial_t u_h^{\alpha_l}\bigr\|_{H^{m-1}(0,T,H_0^1)}^{\frac{1}{m}}.
\end{eqnarray}
Now, we apply the same argument but respect to the parameter. For any $j$, $j=0,\ldots,M$ let us denote
$$
\overline\sigma^{j}_{u_h}=\frac{1}{L+1}\sum_{l=0}^L\|\nabla (I-P^r)u_h^{\alpha_l}(t_j)\|_0^2.
$$
Arguing as before, one can prove that for any $j=0,\ldots,M$, there exists a constant $C$ such 
that the following bound holds
\begin{align}\label{poin2}
\max_{0\le l\le L} 
\bigl\|\nabla (I-P^r)u_h^{\alpha_l}(t_j)\|_0^2 \le& C \left(\overline\sigma^{j}_{u_h}\right)^{{1}-\frac{1}{2m}}\bigl\| (I-P^r)\partial_\mu u_h^{\mu}(t_j)\bigr\|_{H^{m-1}(\alpha_0,\alpha_L,H_0^1)}^{\frac{1}{m}}.
\end{align}
Then
\begin{align*}
\overline\sigma^{\alpha_l}_{u_h}&=\frac{1}{M+1}\sum_{j=0}^M\|\nabla (I-P^r)u_h^{\alpha_l}(t_j)\|_0^2
\le \frac{1}{M+1}\sum_{j=0}^M \max_{0\le l\le L} 
\bigl\|\nabla (I-P^r)u_h^{\alpha_l}(t_j)\|_0^2\nonumber\\
&\le \frac{C}{M+1}\sum_{j=0}^M \left(\left(\frac{1}{L+1}\sum_{l=0}^L\|\nabla (I-P^r)u_h^{\alpha_l}(t_j)\|_0^2.\right)^{{1}-\frac{1}{2m}}C_{\mu,u}^j\right),
\end{align*}
where
$$
C_{\mu,u}^j=\bigl\| (I-P^r)\partial_\mu u_h^{\mu}(t_j)\bigr\|_{H^{m-1}(\alpha_0,\alpha_L,H_0^1)}^{\frac{1}{m}}.
$$
Applying discrete Holder inequality followed by \eqref{39} we get
\begin{align*}
\overline\sigma^{\alpha_l}_{u_h}
&\le C\left(\frac{1}{N}\sum_{j=0}^M \sum_{l=0}^L\|\nabla (I-P^r)u_h^{\alpha_l}(t_j)\|_0^2\right)^{{1}-\frac{1}{2m}}\max_{0\le j\le M}C_{\mu,u}^j\nonumber\\
&\le C\left(\sum_{k={r+1}}^{\tilde d_r}\tilde \lambda_k\right)^{{1}-\frac{1}{2m}}\max_{0\le j\le M}C_{\mu,u}^j.
\end{align*}
Going back to \eqref{poin1} we finally obtain for any $0\le j\le M$
\begin{eqnarray}\label{poin3} 
&&\|\nabla (I-P^r)u_h^{\alpha_l}(t_j)\|_0^2 \\
&&\le C \left(\sum_{k={r+1}}^{\tilde d_r}\tilde \lambda_k\right)^{\left({{1}-\frac{1}{2m}}\right)^2}\max_{0\le j\le M}(C_{\mu,u}^j)^{1-\frac{1}{2m}}
\bigl\| (I-P^r)\partial_t u_h^{\alpha_l}\bigr\|_{H^{m-1}(0,T,H_0^1)}^{\frac{1}{m}}.\nonumber
\end{eqnarray}
So that we finally reach
\begin{eqnarray}\label{sup_estandar}
\max_{0\le j\le M, 0\le l\le L}\|\nabla(I-P^r)u_h^{\alpha_l}(t_j^l)\|_0^2\le C c_{ML}\left(\sum_{k={r+1}}^{\tilde d_r}\tilde \lambda_k\right)^{1-\gamma_m},
\end{eqnarray}
where
\begin{eqnarray}
\gamma_m&=&\frac{1}{m}-\frac{1}{4m^2},\label{gamma_m}\\
 c_{ML}&=&\max_{0\le j\le M}(C_{\mu,u}^j)^{1-\frac{1}{2m}}
\max_{0\le l\le L}\bigl\| (I-P^r)\partial_t u_h^{\alpha_l}\bigr\|_{H^{m-1}(0,T,H_0^1)}^{\frac{1}{m}}\nonumber.
\end{eqnarray}
Comparing \eqref{sup_estandar} with \eqref{eq_point_2d} we observe that tail of the eigenvalues~$\tilde \Sigma_r^2=\sum_{k={r+1}}^{\tilde d_r}\tilde \lambda_k$ has exponent $1-\frac{1}{\gamma_m}$  instead of 1, and that, the smoother (with respect to~$t$ and~$\alpha$) the finite-element approximations are  (used to define the snapshots) the closer the exponent  gets to~1.

Once we have obtained a uniform bound for the error in the snapshots, \eqref{sup_estandar}, we can follow the error analysis in \cite[Theorem 2.3]{Bosco_pointwise} to prove a bound for the finite differences in time since the bound in \cite[Theorem 2.3]{Bosco_pointwise} only requires \eqref{sup_estandar}. Then, 
one can prove that for any $l=0,\ldots,L$, there exists a constant $C$ such 
that the following bound holds
\begin{eqnarray}\label{deriesta}
&&\Delta t \sum_{n=1}^M\|\nabla (I-P^r)D^t(u_h^{\alpha_l}(t_n))\|_0^2\nonumber\\
&&\quad\le C\left(c_{ML}\left(\sum_{k={r+1}}^{\tilde d_r}\tilde \lambda_k\right)^{1-\gamma_m}\right)^{{1}-\frac{1}{m}}\bigl\| (I-P^r)\partial_t u_h^{\alpha_l}\bigr\|_{H^{m-1}(0,T,H_0^1)}^{\frac{2}{m}}\nonumber\\
&&\quad \le Cc_{ML}^{1-\frac{1}{m}}\left(\sum_{k={r+1}}^{\tilde d_r}\tilde \lambda_k\right)^{1-\beta_m}\bigl\| (I-P^r)\partial_t u_h^{\alpha_l}\bigr\|_{H^{m-1}(0,T,H_0^1)}^{\frac{2}{m}},
\end{eqnarray}
where 
\begin{equation}\label{beta_m}
\beta_m=\gamma_m+\frac{1}{m}-\frac{\gamma_m}{m}.
\end{equation}
The exponent in the tail of the eigenvalues~$\tilde \Sigma_r^2$ in \eqref{deriesta} is now $1-\frac{1}{\beta_m}$,  instead of 1, and as before, the smoother finite-element approximations are, the closer the exponent gets to~1.

Finally, it is not difficult to check that the error analysis applied in Section \ref{sec4} to get the a priori bound for the new method in Theorem \ref{main} can be reproduced for the standard method applying \eqref{sup_estandar} and \eqref{deriesta} since these are the only bounds needed regarding the POD method. Then, one can prove the following bound for the standard method.
\begin{Theorem}\label{main2}
Assume $u_r^{\alpha}(t)=P^ru_h(0)$ and let $\gamma_m$ and $\beta_m$ be the constants defined in \eqref{gamma_m} and \eqref{beta_m}. For all $t\in[0,T^\alpha]$ and $\alpha\in[\alpha_0,\alpha_L]$, there exists a constant $C$
depending on $T$, $\alpha_L-\alpha_0$ and $C_p$  such that the following bound holds for $q\ge2$, $m\ge2,$ whenever the functions $u_h^{\alpha}$ are smooth enough so that all the terms are well defined 
\begin{eqnarray}\label{max_puntual_todo_sta}
\|u_r^{\alpha}(t)-u_h^{\alpha}(t)\|_0^2&\le& 
C_{m,q}\left(\left(\sum_{k={r+1}}^{d_r}\lambda_k\right)^{1-\gamma_m}+\left(\sum_{k={r+1}}^{d_r}\lambda_k\right)^{1-\beta_m}\right)\nonumber\\
&&\ +C_{u_h}^1(\Delta t)^{2q}+C_{u_h}^2(\Delta \alpha)^{2m-1},
\end{eqnarray}
where $C_{m,q}=C((m+1)+qm)$ and the constants $C_{u_h}^1$ and $C_{u_h}^2$ depend on some norms of some derivatives of $u_h$, as in \eqref{max_puntual_todo}, and the constant $C$ depends also on some norms of some derivatives of $u_h$ according to \eqref{sup_estandar} and \eqref{deriesta}. 
\end{Theorem}

\section{Continuous-in-time case with two parameters}\label{sec6}
In this section we show how to extend the definition of the new method to the case in which the approximations depend on two parameters. To simplify the exposition we assume we have numerical approximations $u_h^{\alpha_l,\beta_k}$, $l=0,\ldots,L$, $k=0,\ldots,S$ defined all at the same time interval $[0,T]$. The notation $u_h^{\alpha_l,\beta_k}$ means that the finite element approximation \eqref{eq:model} depends now on two parameters which can be chosen from the following: the  diffusion parameter, parameters in the nonlinear term, parameters in the forcing term or parameters in the initial condition.

For a fixed integer $L>0$, $\Delta \alpha=(\alpha_L-\alpha_0)/L$ and $\alpha_l=\alpha_0+l\Delta \alpha$, $l=0,\ldots,L$. 
For a fixed integer $S>0$, $\Delta \beta=(\beta_S-\beta_0)/S$ and $\beta_k=\beta_0+k\Delta \beta$, $k=0,\ldots,S$. 
We take $\Delta t=T/M$, for a fixed integer $M>0$.

In this case, as we will see in the definition of the set $\cal U$ below, the method is based on first, second and third order difference quotients.

We will denote by $D^t$ the finite difference respect to time and by $D^\alpha$ and $D^\beta$ the finite differences respect to $\alpha$ and $\beta$, respectively. This means that
\begin{eqnarray*}
&&D^t v^{\alpha_l,\beta_k}(t_j)=\frac{v^{\alpha_l,\beta_k}(t_j)-v^{\alpha_l,\beta_k}(t_{j-1})}{\Delta t},\\ 
&&D^\alpha v^{\alpha_l,\beta_k}(t_j)= \frac{v^{\alpha_l,\beta_k}(t_j)-v^{\alpha_{l-1},\beta_k}(t_{j})}{\Delta \alpha},\ 
D^\beta v^{\alpha_l,\beta_k}(t_j)= \frac{v^{\alpha_l,\beta_k}(t_j)-v^{\alpha_{l},\beta_{k-1}}(t_{j})}{\Delta \beta},
\end{eqnarray*}
where $v^{\alpha_l,\beta_k}:[0,T]\rightarrow V_h^k$, $l=0,\ldots,L$, $k=0,\ldots,S$.
For $N=(M+1)(L+1)(S+1)$ we define $\bU=\hbox{\rm span}({\cal U})$ where
\begin{eqnarray*}
\label{labU}
{\cal U}&=& \left\{\sqrt{N}u_h^{\alpha_l,\beta_k}(t_0),\ 0\le l\le L,\ 0\le k\le S,\right.\nonumber\\
&&\left.\sqrt{(L+1)(S+1)}D^t u_h^{\alpha_0,\beta_k}(t_j),\ 1\le j\le M,\ 0\le k\le S,\right.\nonumber\\
&&\left. \sqrt{(S+1)}D^t D^\alpha u_h^{\alpha_l,\beta_0}(t_j), 1\le j\le M,\ 1\le l\le L,\right.\nonumber\\
&&\left. D^t D^\alpha D^\beta u_h^{\alpha_l,\beta_k}(t_j), 1\le j\le M,\ 1\le l\le L,\ 1\le k\le S\right\}.
\end{eqnarray*}
Let us observe that the set ${\cal U}$ has $N=(M+1)(L+1)(S+1)$ elements. Keeping the same notation as before, for this new method
the following bound holds:
\begin{eqnarray}\label{cota_pod_b_2}
\sum_{l=0}^L\sum_{k=0}^S\|\nabla(I-P^r)u_h^{\alpha_l,\beta_k}(t_0)\|_0^2
+\frac{1}{M+1}\sum_{j=1}^M\sum_{k=0}^S\|\nabla(I-P^r)D^t u_h^{\alpha_0,\beta_k}(t_j)\|_0^2\nonumber\\
\ +\frac{1}{(M+1)(L+1)}\sum_{j=1}^M\sum_{l=1}^L\|\nabla(I-P^r)D^t D^\alpha (u_h^{\alpha_l,\beta_0}(t_j)\|_0^2
\nonumber\\
\ +\frac{1}{N}\sum_{j=1}^M\sum_{l=1}^L\sum_{k=1}^S \|\nabla(I-P^r)D^tD^\alpha D^\beta u_h^{\alpha_l,\beta_k}(t_j)\|_0^2=\sum_{k={r+1}}^{d_r}\lambda_k.
\end{eqnarray}
Next, we prove that pointwise estimates in time can also be proved for the method, analogous to \eqref{eq_point_2d}, see \eqref{eq_point_3d}. As in Section \ref{sec4}, we need to prove a previous lemma. 
\begin{lema} \label{lema1_2} Le $X$ be a Banach space and let $z(t_j,\alpha_l,\beta_k)\in X$, $0\le j\le M$, $0\le l\le L$, $0\le k\le S$ then
\begin{eqnarray}\label{inf_pri_2}
&&\|z(t_j,\alpha_l,\beta_k)\|_X^2\le 4\|z(t_0,\alpha_l,\beta_k)\|_X^2+4(j\Delta t)\sum_{s=1}^j( \Delta t)\|D^t z(t_s,\alpha_0,\beta_k)\|_X^2 
\nonumber\\
&&\quad +4 (j\Delta t)(l\Delta \alpha)\sum_{s=1}^j\sum_{n=1}^l (\Delta\alpha)(\Delta t)\|D^t D^\alpha  z(t_s,\alpha_n,\beta_0)\|_X^2 \nonumber\\
&&\quad+4 (j\Delta t)(l\Delta \alpha)(k\Delta \beta) \sum_{s=1}^j\sum_{n=1}^l\sum_{r=1}^k
(\Delta \beta)(\Delta \alpha)(\Delta t) \|D^t D^\alpha D^\beta z(t_s,\alpha_n,\beta_r)\|_X^2.
\end{eqnarray}
\end{lema}
\begin{proof}
We first observe that 
\begin{eqnarray}\label{unodedos}
z(t_j,\alpha_l,\beta_k)=z(t_0,\alpha_l,\beta_k)+\sum_{s=1}^jD^t z(t_s,\alpha_l,\beta_k)\Delta t.
\end{eqnarray}
Now, arguing as in Lemma \ref{lema1} (with parameters $\alpha$ and $\beta$ playing the role of $t$ and $\alpha$ and
fixing $t_s$) we get
\begin{eqnarray}\label{dosdedos}
z(t_s,\alpha_l,\beta_k)=z(t_s,\alpha_0,\beta_k)+\sum_{n=1}^lD^\alpha z(t_s,\alpha_n,\beta_0)\Delta \alpha+\sum_{n=1}^l\sum_{r=1}^k
D^\alpha D^\beta z(t_s,\alpha_n,\beta_r)\Delta \alpha \Delta \beta.
\end{eqnarray} 
Inserting \eqref{dosdedos} into \eqref{unodedos} we obtain
\begin{eqnarray*}
z(t_s,\alpha_l,\beta_k)&=&z(t_0,\alpha_l,\beta_k)+\sum_{s=1}^j D^t z(t_s,\alpha_0,\beta_k)\Delta t 
+\sum_{s=1}^j\sum_{n=1}^lD^t D^\alpha z(t_s,\alpha_n,\beta_0)\Delta \alpha \Delta t\nonumber\\
&&\quad +\sum_{s=1}^j\sum_{n=1}^l\sum_{r=1}^k D^t D^\alpha D^\beta z(t_s,\alpha_n,\beta_r)\Delta \alpha \Delta \beta \Delta t.
\end{eqnarray*}
Taking norms
\begin{eqnarray*}
&&\|z(t_s,\alpha_l,\beta_k)\|_X=\|z(t_0,\alpha_l,\beta_k)\|_X+\sum_{s=1}^j\| D^t z(t_s,\alpha_0,\beta_k)\|_X\Delta t 
\nonumber\\
&&\ +\sum_{s=1}^j\sum_{n=1}^l\|D^t D^\alpha z(t_s,\alpha_n,\beta_0)\|_X\Delta \alpha \Delta t +\sum_{s=1}^j\sum_{n=1}^l\sum_{r=1}^k \|D^t D^\alpha D^\beta z(t_s,\alpha_n,\beta_r)\|_X\Delta \alpha \Delta \beta \Delta t.
\end{eqnarray*}
And then
\begin{eqnarray*}
\|z(t_s,\alpha_l,\beta_k)\|_X^2&\le& 4\|z(t_0,\alpha_l,\beta_k)\|_X^2+4\left(\sum_{s=1}^j\| D^t z(t_s,\alpha_0,\beta_k)\|_X\Delta t\right)^2 
\nonumber\\
&&\ +4\left(\sum_{s=1}^j\sum_{n=1}^l\|D^t D^\alpha z(t_s,\alpha_n,\beta_0)\|_X\Delta \alpha \Delta t \right)^2
\nonumber\\
&&\ +4\left(\sum_{s=1}^j\sum_{n=1}^l\sum_{r=1}^k \|D^t D^\alpha D^\beta z(t_s,\alpha_n,\beta_r)\|_X\Delta \alpha \Delta \beta \Delta t\right)^2.
\end{eqnarray*}
We finally reach \eqref{inf_pri_2} applying discrete Cauchy-Schwarz inequality.
\end{proof}
\begin{lema} \label{lema2_2}
The following bound holds
\begin{eqnarray}\label{eq_point_3d}
\max_{0\le j\le M, 0\le l\le L, 0\le k\le S}\|(I-P^r)u_h^{\alpha_l,\beta_k}(t_j)\|_X^2\le C_X \sum_{k={r+1}}^{d_r}\lambda_k,
\end{eqnarray}
for $C_X=C_{H_0^1}:=4\max\left(1,2T^2,4T^2(\alpha_L-\alpha_0)^2,8T^2(\alpha_L-\alpha_0)^2(\beta_S-\alpha_0)^2\right)$ if $X=H_0^1(\Omega)$ and
$C_X=C_{L^2}:=C_p^2 C_{H_0^1}$ if $X=L^2(\Omega)$.
\end{lema}
\begin{proof}
We take $z(t_j,\alpha_l,\beta_k)=u_h^{\alpha_l,\beta_k}(t_j)-P^r u_h^{\alpha_l,\beta_k}(t_j)$ and apply \eqref{inf_pri_2} from Lemma \ref{lema1_2}.
Then
\begin{eqnarray*}
&&\|(I-P^r)u_h^{\alpha_l,\beta_k}(t_j)\|_X^2\le 4\|(I-P^r)u_h^{\alpha_l,\beta_k}(t_0)\|_X^2\nonumber\\
&&\ +4(j\Delta t)\sum_{s=1}^j( \Delta t)\|(I-P^r)D^t u_h^{\alpha_0,\beta_k}(t_s)\|_X^2 
\nonumber\\
&&\ +4 (j\Delta t)(l\Delta \alpha)\sum_{s=1}^j\sum_{n=1}^l (\Delta\alpha)(\Delta t)\|(I-P^r)D^t D^\alpha  u_h^{\alpha_n,\beta_0}(t_s)\|_X^2 \nonumber\\
&&\ +4 (j\Delta t)(l\Delta \alpha)(k\Delta \beta) \sum_{s=1}^j\sum_{n=1}^l\sum_{r=1}^k
(\Delta \beta)(\Delta \alpha)(\Delta t) \|(I-P^r)D^t D^\alpha D^\beta u_h^{\alpha_n,\beta_r}(t_s)\|_X^2.
\end{eqnarray*}
Taking $L\ge 1$, $M\ge 1$, $S\ge 1$ so that $(L+1)/L\le 2$, $(M+1)/M\le 2$ and $(S+1)/S\le 2$ and denoting by
$C_{H_0^1}=4\max\left(1,2T^2,4T^2(\alpha_L-\alpha_0)^2,8T^2(\alpha_L-\alpha_0)^2(\beta_S-\alpha_0)^2\right)$ we get
\begin{align*}
\|(I-P^r)u_h^{\alpha_l,\beta_k}(t_j)\|_X^2\le C_{H_0^1}\left(\|(I-P^r)u_h^{\alpha_l,\beta_k}(t_0))\|_X^2\right.\nonumber\\
\ +\frac{1}{M+1}\sum_{s=1}^j\|(I-P^r)D^t u_h^{\alpha_0,\beta_k}(t_s)\|_X^2 
\nonumber\\
\ +\frac{1}{(M+1)(L+1)}\sum_{s=1}^j\sum_{n=1}^l \|(I-P^r)D^t D^\alpha  u_h^{\alpha_n,\beta_0}(t_s)\|_X^2 \nonumber\\
\ \left.+\frac{1}{N} \sum_{s=1}^j\sum_{n=1}^l\sum_{r=1}^k
 \|(I-P ^r)D^t D^\alpha D^\beta u_h^{\alpha_n,\beta_r}(t_s)\|_X^2\right).
\end{align*}
Taking $X=H_0^1(\Omega)^d$ and applying \eqref{cota_pod_b_2} we get
\begin{equation}\label{here2}
\|\nabla(I-P^r)u_h^{\alpha_l,\beta_k}(t_j)\|_0^2\le C_{H_0^1} \sum_{k={r+1}}^{d_r} \lambda_k.
\end{equation}
Since the above inequality is valid for any $0\le j\le M$, $0\le l\le L$ and $0\le k\le S$, taking the maximum we reach \eqref{eq_point_3d} for $X=H_0^1(\Omega)$. In the case $X=L^2(\Omega)$ we conclude applying Poincar\'e inequality \eqref{poincare} to \eqref{here2} and taking the maximum.
\end{proof}
The rest of the analysis of the method can be carried out as in Section \ref{sec4}. We do not include the details here to avoid further increasing the length of the paper. Also, the ideas in Section \ref{sec5}, can be extended to this case, so that the error analysis of the standard method with two parameters can also be handled in a similar way. 

\section{Numerical Experiments}\label{sec7}

We consider the system of the Brusselator with diffusion in one spatial dimension
\begin{equation}
\label{bruss1D}
\begin{array}{rclcl}
y_t& =& \nu y_{xx} + y^2z - (\beta+1)y + \alpha,& \quad& (x,t)\in(0,1)\times (0,T],\\
z_t &=& \nu z_{xx} + \beta y -  y^2z ,&&  (x,t)\in(0,1)\times (0,T],\\
&&y_x(0,t) = z_x(0,t)=0,&& t\in(0,T],\\
&&y(1,t)=\alpha,\quad z(1,t)=\beta/\alpha, &&  t\in(0,T],\\
&&y(x,0)=y_0(x),\quad z(x,0)=z_0(x), && x\in(0,1).
\end{array}
\end{equation}
There is a steady state solution $y=\alpha$, $z=\beta/\alpha$ which, for certain values of~$\alpha$ and~$\beta$, is unstable and a (stable) periodic orbit develops, 
so that the system possesses asymptotic nontrivial dynamics, as opposed to many reaction-diffusion systems where solutions tend to a steady state as time tends to infinity. For $\nu=0$, stable periodic orbits exist only for $\beta >1+\alpha^2$, and for a smaller set of values $(\alpha,\beta)$ for $\nu>0$.  We remark that, in our opinion, it makes more sense to get the data set for POD methods in the attractor to which all solutions converge rather than in transients, and stable periodic orbits and invariant tori, with their nontrivial dynamics, are good candidates to test numerical methods.

In order to work with homogeneous boundary conditions, we rewrite the solution as
$$
\begin{array}{rcl} y &=& \alpha + u\\
z&=&\frac{\beta}{\alpha}+v,
\end{array}
$$
so that the steady state solution corresponds to~$u=v=0$. In the sequel, we detone by~$\bu=(u,v)$ the two-component solution of~\eqref{bruss1D} in the new variables~$u$ and~$v$.

We comment first of the selection of parameters, which was guided by the case with two parameters commented at the end of the present section. For $\nu=0.1$ and~$\alpha=1$, the stable orbits exists only for $\beta\in [2.4539,5.2979]$. Also, for $\beta$ in the vicinity of~$2.4539$, convergence to the periodic orbit is slow, so that we decided to consider only $\beta\ge 2.75$. On the other hand,  for $\beta>4$, the periodic solutions present large variations in a small part of the period while changing little in the rest of the period. This can be seen in~Fig.~\ref{fig_4p5} where we show the $L^2$ norm of~$\bu_{h,t}$ for~$\beta=4.5$ and $\nu=0.01$. Large variations can be seen near the endpoints of the period, while in the rest of the period the time derivative has a moderate size. In this situation, we have found that POD methods based on snapshots taken at equally-spaced times as those considered in the present paper need large values of~$r$ to produce accurate approximations.  A better alternative could be to concentrate snapshots in those parts with larger variations (see \cite{temporal_nos}). This will be subject of future research. Since in the present paper we confine ourselves to equally-spaced times, we decided to consider only~$\beta\le 4.25$.
\begin{figure}[h]
\begin{center}
\includegraphics[height=2.9truecm]{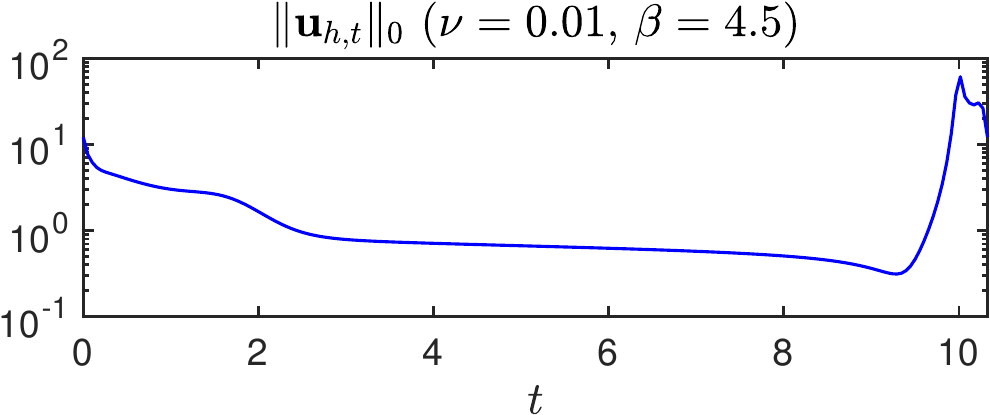}
\caption{\label{fig_4p5} Variation of the time derivative over a period for~$\nu=0.01$ and~$\beta=4.5$}
\end{center}
\end{figure}

We first present the results for the POD method for a system with one parameter. For this purpose, we consider~\eqref{bruss1D} with~$\nu=0.01$, $\alpha=1$ and~$\beta\in[2.75,4.25]$, where the periods of~the family of periodic solutions range from~$T=6.7725$ for $\beta=2.75$ to~$T=9.5949$ for $T=4.25$ (both values rounded to five significant digits). 
For the FEM approximation in~space, we consider quadratic elements on a uniform partition of~$[0,1]$ into elements of length~$h=1/50$. For the time integration we used the numerical differentiation  formulae (NDF) \cite{NDF1}, 
in the variable-step, variable-order implementation of {\sc Matlab}'s command {\tt ode15s} \cite{odesuite}, with tolerance values $10^{-8}$ and $10^{-11}$ for relative and absolute values, respectively, of the local error.
To compute a periodic solution of a FEM approximation to~\eqref{bruss1D} for a given value of~$\beta$, we proceeded as follows.
Taking a small perturbation of the unstable equilibrium as initial condition, we computed the corresponding solution and the times where the $L^2$-norm of the solution reached a local maximum (these were found as ceros of the function $t\mapsto (\bu_h,\bu_{h,t})$), the difference between two consecutive values (sufficiently large to avoid including other local minima inside a period) being an approximation to the period for sufficiently large~$t$; the solution was computed until two consecutive approximations to the period differ in less that $10^{-7}$ (in
relative size). The corresponding value $\bu_h$ of the solution was kept as initial condition for the periodic orbit.  We did this for $\beta=2.75$, $3.25$, $3.75$
and~$4.25$, repeating the computations with $h=1/100$ and comparing the results to get an idea of the error committed by the FEM method with $h=1/50$. The maximum of these errors, which was achieved for~$\beta=4.25$, was $1.04\times 10^{-4}$ in the $L^2$ norm and~$1.73\times 10^{-3}$ in the $H^1$ norm.

For the POD method, we took $M=64$, that is, for every $\beta=2.75, 3.25, 3.75, 4.25$ we took $65$ snapshots on equally distributed times $t_n^{\beta}$ on each period~$[0,T^\beta]$, and
the corresponding POD basis and POD approximations~$\bu_r$ were computed as described in~Section~\ref{sec3} for $r=18$, $24$, $30$ and~$36$.
In Table~\ref{table1}, besides showing the quantity~$\Sigma_r$ defined in~\eqref{the_sigma_r} in the second column, we show,  on the third and fourth columns, the maximum of the errors
\begin{equation}
\label{error_notation}
{\boldsymbol \varepsilon}_r(t) = (I-P_r)\bu_h(t),\qquad \be_r(t) = \bu_h(t)-\bu_r(t),
\end{equation}
 respectively,
in the $H^1$ norm (in all, cases, achieved for~$\beta=4.25$), measured on a very fine partition over each period (2049 equally-distributed points). 
\begin{table}[h]
$$
\begin{array}{|c|c|c|c|c|c|}
\hline 
r & \Sigma_r &
 {\displaystyle \max_{\beta_k,t}}\|\nabla {\boldsymbol \varepsilon_r}(t) \|_0  &
 {\displaystyle \max_{\beta_k,t}}\|\nabla \be_r(t) \|_0 & {\displaystyle \max_{\beta_k,t_n}}\|\nabla {\boldsymbol \varepsilon_r}(t_n) \|_0  &
{\displaystyle \max_{\beta_k,t_n}}\|\nabla \be_r(t_n) \|_0
 \\ \hline
18 & 1.43{\rm e}{-2} & 3.32{\rm e}{-1} & 4.63{\rm e}{-1} &1.10{\rm e}{-2} & 3.83{\rm e}{-1} \\ \hline
 24 & 1.40{\rm e}{-3} & 1.40{\rm e}{-1} & 1.59{\rm e}{-1} & 1.52{\rm e}{-3} & 1.77{\rm e}{-2} \\ \hline
30 & 5.22{\rm e}{-5} & 2.58{\rm e}{-2} & 2.90{\rm e}{-2} & 5.61{\rm e}{-5} & 3.62{\rm e}{-3} \\ \hline
36 & 2.65{\rm e}{-6} & 4.98{\rm e}{-3} & 5.53{\rm e}{-3} & 3.41{\rm e}{-6} & 5.69{\rm e}{-4} \\ \hline
\end{array}
$$
\caption{\label{table1} Maximum errors ${\boldsymbol\varepsilon}_r$ and~$\be_r$ for  $M=64$ and~$\beta=2.75$, $3.25$, $3.75$
and~$4.25$.}
\end{table}
It can be noticed that  both~${\boldsymbol\varepsilon}_r$ and~$\be_r$ have very similar values, and that these are much larger than the corresponding values of~$\Sigma_r$. Yet if we look at the error on the times~$t_n$ where the snapshots were taken (last two columns of Table~\ref{table1}), we see that the errors decrease significantly and that the values of ${\boldsymbol \varepsilon_r}$ are very much in line with those of~$\Sigma_r$. The explanation for the disparity between the maximum errors when measured on all values of time over a period or only on those corresponding to the snapshots of the dataset can be seen in Figure~\ref{fig_err_beta}, where, on the top plots, for $\beta=4.25$, $r=30$ and $M=32$, the values~$\|\be_r(t)\|_0$ are shown in a continuous blue line, those of~$\|{\boldsymbol\varepsilon}_r(t)\|_0$ on a discontinuous magenta line, those of~$\|\be_r(t_n)\|_0$  with a red circle and those
of~$\|{\boldsymbol\varepsilon}_r(t_n)\|_0$ with a red cross. We see that $\be_r$ and~${\boldsymbol\varepsilon_r}$  allways coincide or are very close except at the end of the interval, where they differ significantly only on the values of~$t_n$ (the right-plot is a magnification of the left-one at the end of the period). On the bottom plots we show the errors at the end of the period for $M=64$ and $M=128$, and we notice that, they do not decrease with respect to the case~$M=32$, and that the maximum of the errors over the values $t_n$ actually increases as more of these values of time fall in the part where the largest errors are (this can be confirmed in~Table~\ref{table2}, where we show these maxima).

\begin{figure}[h]
\begin{center}
\includegraphics[height=2.9truecm]{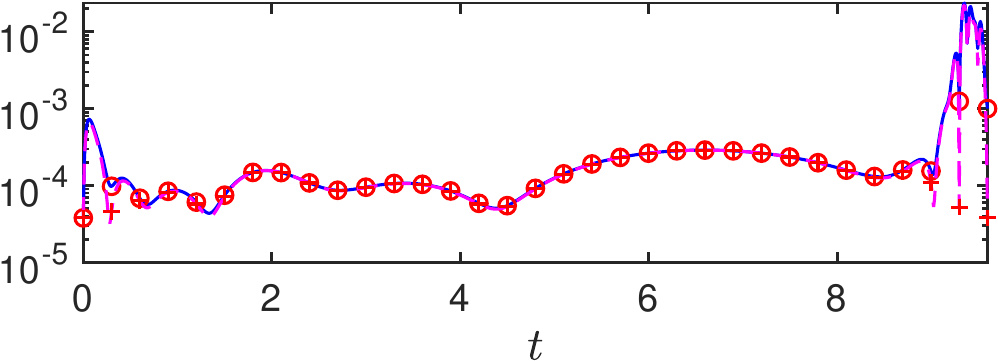} \quad
\includegraphics[height=2.9truecm]{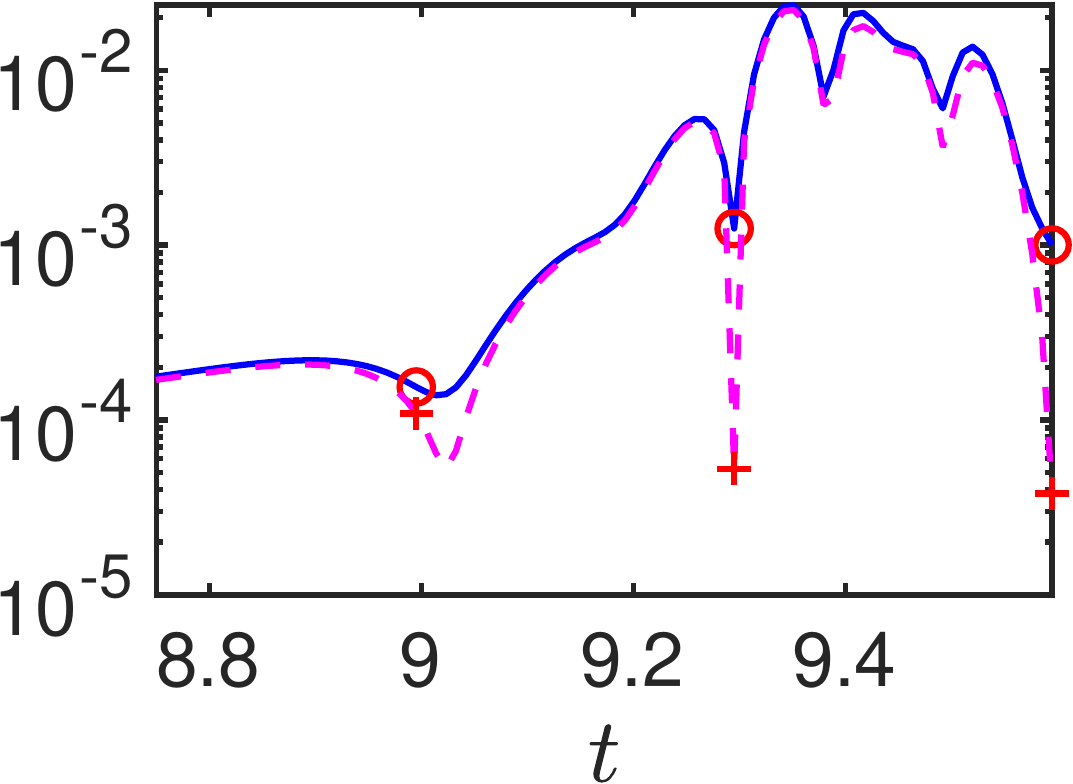} 

\includegraphics[height=2.9truecm]{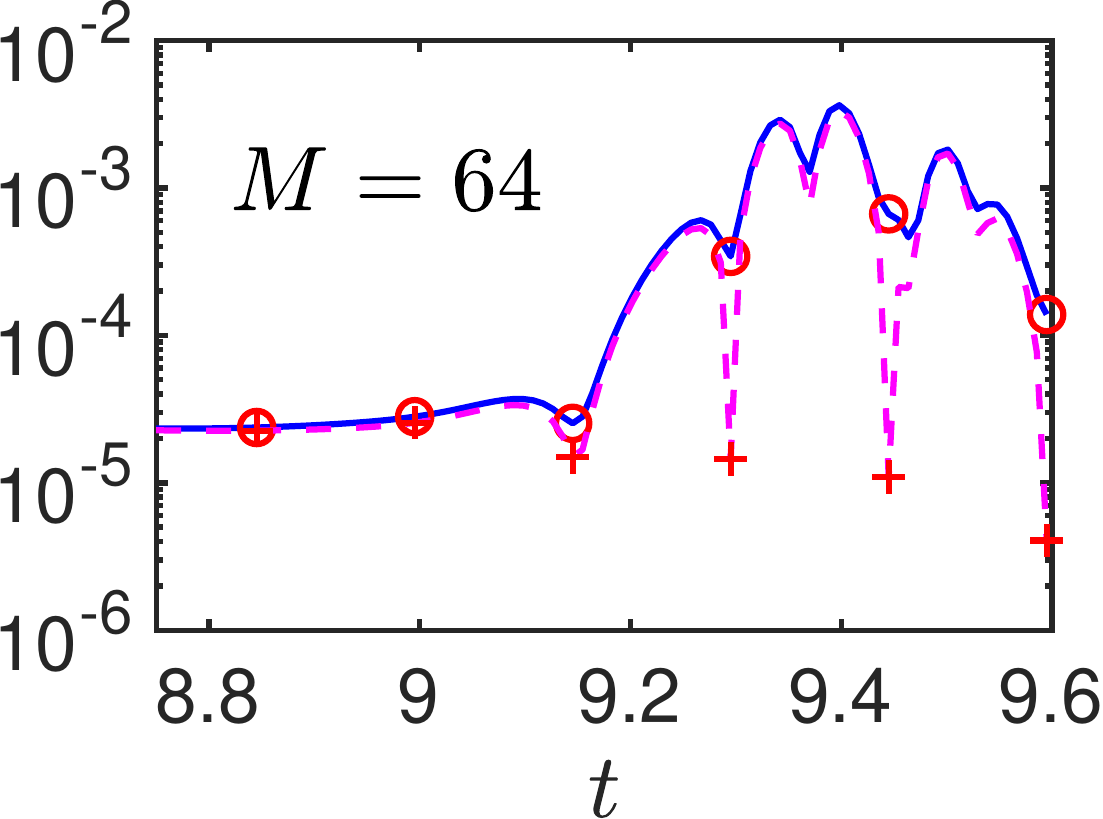} \quad
\includegraphics[height=2.9truecm]{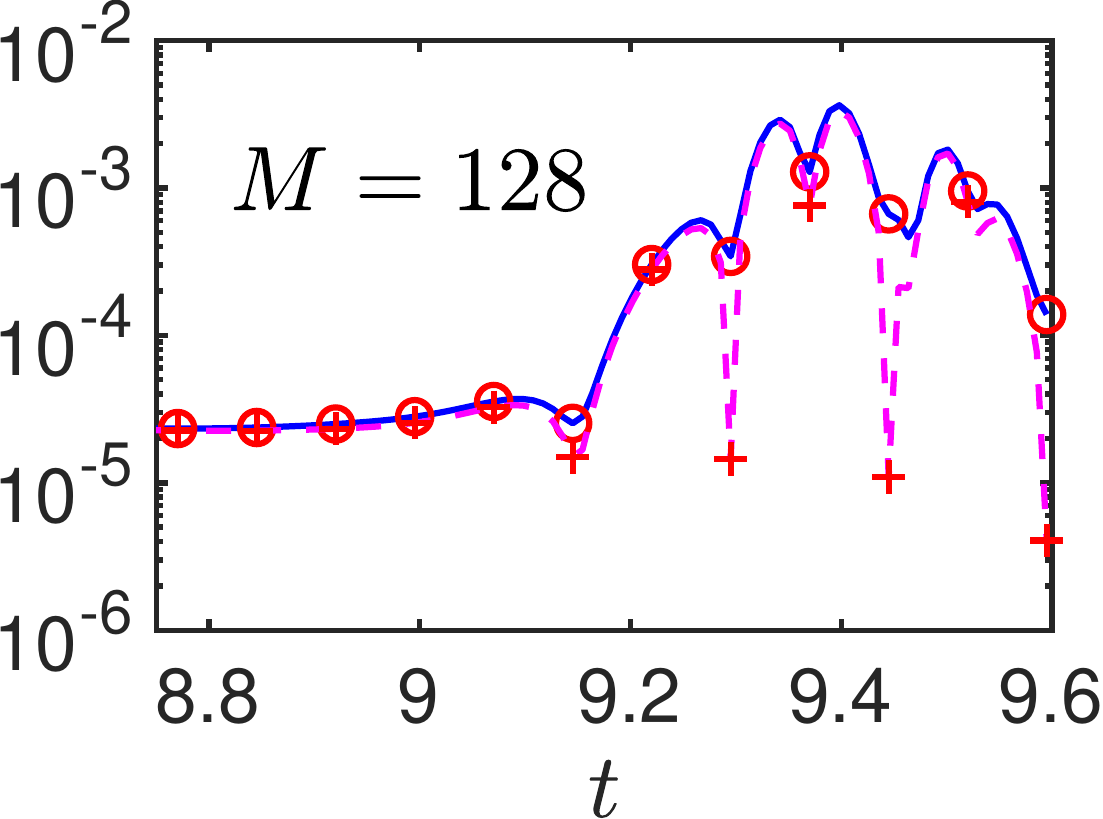} 

\caption{\label{fig_err_beta} Errors $\|\nabla \be_r(t)\|_0$ (continuous blue line), $\|\nabla {\boldsymbol\varepsilon}_r(t)\|_0$ (discontinuous magenta line), of~$\|\nabla \be_r(t_n)\|_0$  (red circles) and~$\|{\boldsymbol\varepsilon}_r(t_n)\|_0$ (red crosses) for $\beta=4.25$. Top: $M=32$ on period $[0,T^\beta]$ (left) and detail at the end of the period (right); bottom: detail at the end of the period for~$M=64$ and~$M=128$.}
\end{center}
\end{figure}
\begin{table}[h]
$$
\begin{array}{|c|c|c|c|c|c|}
\hline 
r &
%\genfrac{}{}{0pt}{{\displaystyle \max_{t_n}}\|\nabla {\boldsymbol \varepsilon_r}(t) \|_0}{(M=32)}  &
%\genfrac{}{}{0pt}{{\displaystyle \max_{t_n}}\|\nabla {\boldsymbol \varepsilon_r}(t) \|_0}{(M=128)}
% \genfrac{}{}{0pt}{{\displaystyle \max_{t_n}}\|\nabla\be_r(t) \|_0}{(M=32)}  &
%\genfrac{}{}{0pt}{{\displaystyle \max_{t_n}}\|\nabla  \be_r(t) \|_0}{(M=128)}
\genfrac{}{}{0pt}{0}{{\displaystyle \max_{t_n}}\|\nabla {\boldsymbol \varepsilon_r}(t_n) \|_0}{(M=32)}  &
\genfrac{}{}{0pt}{0}{{\displaystyle \max_{t_n}}\|\nabla {\boldsymbol \varepsilon_r}(t_n) \|_0}{(M=128)} &
\genfrac{}{}{0pt}{0}{{\displaystyle \max_{t_n}}\|\nabla\be_r(t_n) \|_0}{(M=32)}  &
\genfrac{}{}{0pt}{0}{{\displaystyle \max_{t_n}}\|\nabla\be_r(t_n) \|_0}{(M=128)}
 \\ \hline
18  & 5.47{\rm e}{-3} & 2.17{\rm e}{-1} &3.16{\rm e}{-2} & 3.83{\rm e}{-1} \\ \hline
 24 & 4.25{\rm e}{-4} & 5.98{\rm e}{-2} & 1.02{\rm e}{-2} & 8.48{\rm e}{-2} \\ \hline
30 & 6.35{\rm e}{-6} & 1.05{\rm e}{-2} & 5.21{\rm e}{-3} & 1.30{\rm e}{-2} \\ \hline
36  & 4.43{\rm e}{-7} & 4.98{\rm e}{-3} & 7.34{\rm e}{-4} & 5.53{\rm e}{-3} \\ \hline
\end{array}
$$
\caption{\label{table2} Maximum errors ${\boldsymbol\varepsilon}_r$ and~$\be_r$ over times~$t_n=T^\beta/M$ for~$\beta=2.75$, $3.25$, $3.75$
and~$4.25$ and $M=32,128$}
\end{table}

To decrease the errors for~$t=[0,T^\beta]$ one can either increase~$r$ or increase~$L$ (the number of values of~$\beta$). For example, if for
for~$M=64$, we take~$r=42$, then, errors~${\max_{\beta_k,t}}\|\nabla {\boldsymbol \varepsilon_r}(t) \|_0$ and~${\max_{\beta_k,t}}\|\nabla \be_r(t) \|_0$ go down to~$2.43\times 10^{-3}$ and~$2.81\times 10^{-3}$, respectively, very close to
the error in the FEM approximation with $h=1/50$. On the other hand, if for~$M=64$ and $r=36$ we increase $L$ from~3 to~15, then, errors errors~${\max_{\beta_k,t}}\|\nabla {\boldsymbol \varepsilon_r}(t) \|_0$ and~${\max_{\beta_k,t}}\|\nabla \be_r(t) \|_0$ go down to~$1.57\times 10^{-3}$ and~$1.64\times 10^{-3}$, again very close to the error of the FEM approximation with $h=1/50$. 

Once we have found the value of~$r$ for which the POD method reproduces the data set with an accuracy similar to FEM approximations in the data set, we now check its performance out of the data set. For this purpose, we use the POD method with $r=42$ to compute the periodic orbits for $31$ values of~$\beta$ equally spaced in~$[2.75,4.25]$, using the same technique we used above for the data set, that is, starting with a small perturbation of the zero solution, we integrated the POD system in time checking for the times where the $L^2$ norm of~$\bu_r$ achieved a local maximum, taking the difference between two consecutive local maxima (with sufficiently large value) as an approximation to the period and stopping when the relative error between two consecutive approximations to the period differed in less than $10^{-8}$, taking $\bu_r$ at the last of these maxima as the initial condition for the computation of the periodic orbit. Notice that this process is done without any further computation with the FEM method, so that the POD method is used to compute something not previously computed with the FEM method. Then, to check the accuracy of the POD method, we also computed the periods and the initial conditions with the FEM for the same 31 values of~$\beta$. The accuracy of the POD method proved to be excellent in the computation of the periods, since their relative error with respect to the periods of the FEM approximation was below~$3.69\times 10^{-9}$.

\begin{figure}[h]
\begin{center}
\includegraphics[height=5truecm]{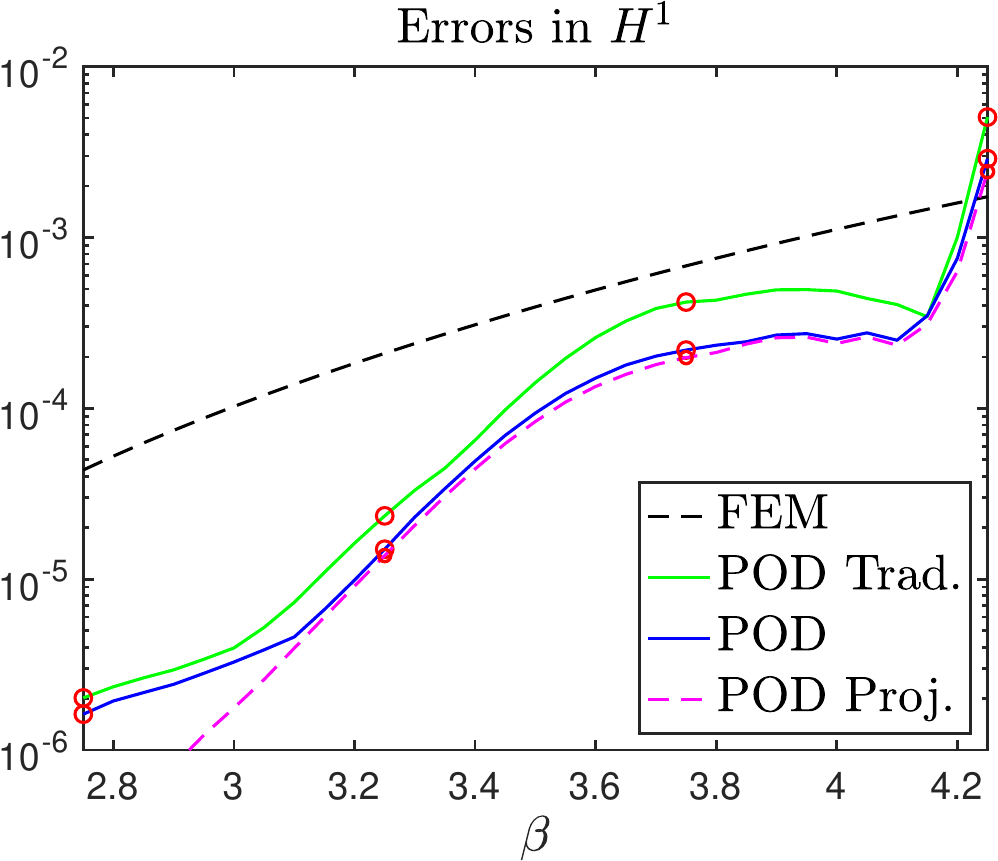}\quad 
\includegraphics[height=5truecm]{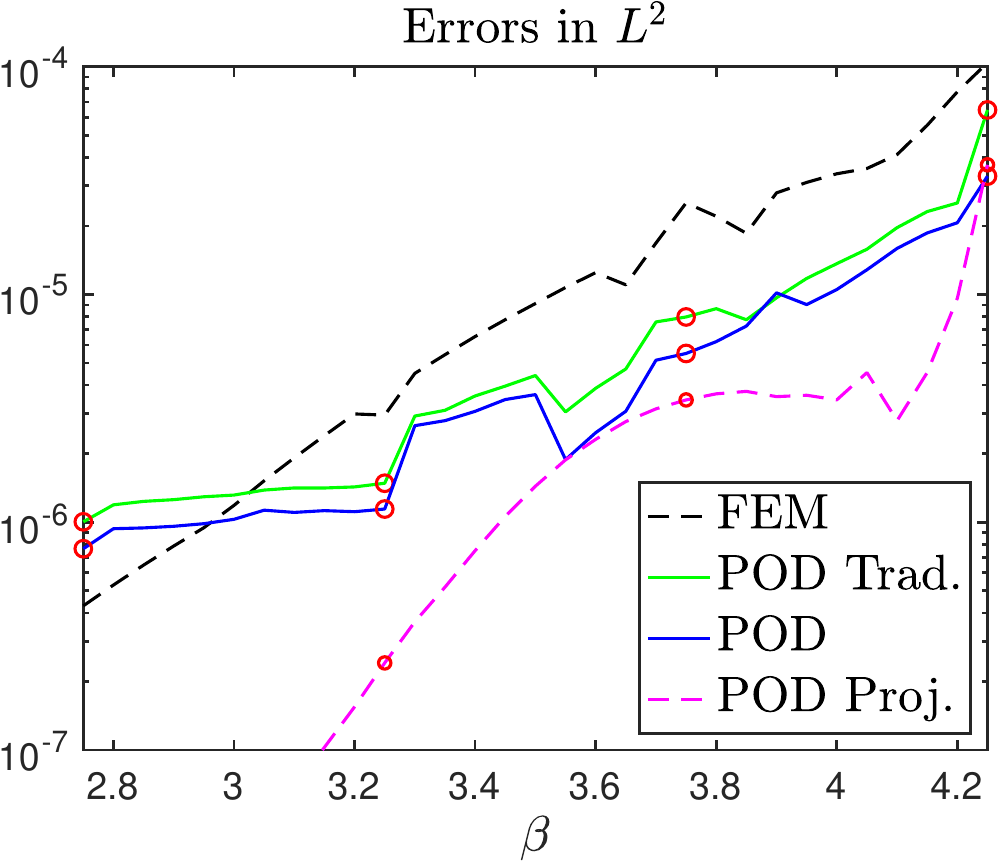}
\caption{\label{fig_range} Left: Errors $\max_{t}\| \nabla  \be_r(t)\|_0$ (blue continuous), $\max_{t}\| \nabla  {\boldsymbol\varepsilon}_r(t)\|_0$ (magenta discontinuous), $\max_{t}\|\nabla( \bu_h(t)-\bu_{h/2} (t))\|_0$ (black discontinuous) and $\max_{t}\| \nabla  \be_r(t)\|_0$ (green continuous) for the traditional POD method; circles mark the values of $\beta$ used for the dataset. Right:  Errors $\max_{t}\|   \be_r(t)\|_0$ (blue continuous), $\max_{t}\|{\boldsymbol\varepsilon}_r(t)\|_0$ (magenta discontinuous), $\max_{t}\| \bu_h(t)-\bu_{h/2} (t)\|_0$ (black discontinuous) and $\max_{t}\|  \be_r(t)\|_0$ (green continuous) for the traditional POD method.}
\end{center}
\end{figure}
With respect to the accuracy of the approximation in the whole period, the results can be seen in~Fig.~\ref{fig_range} where, for each of the 31 values of beta, we show the maximum of the error on a very fine partition (2048 elements) of the corresponding period. We see that the  error $\be_r$ remains well below the 
FEM error $\bu_h-\bu_{h/2}$ for most of the values of~$\beta$, the exception being the largest values of~$\beta$ in the case of the $H^1$ norm and the smallest ones in the case of~$L^2$, which suggests that~$r$ should be chosen for each~$\beta$ according to the (known or estimated) accuracy of the FE space where computations are carried out. We also see that, at least for $H^1$-errors, that $\be_r$ is very close to the projection error~${\boldsymbol \varepsilon}_r$ (magenta line). Finally, we show the error~$\be_r$ of the traditional POD method (green line), which in this example is slightly larger than that of POD method introduced in the present paper. We must say that this is not always the case, since for smaller values of~$r$ (i.e., lower accuracy) the traditional POD method produced slightly smaller errors than the new one.

We now comment on the two-parameter case. For this purpose, we set $\rho=-\log_{10}(\nu)$ as a second parameter and take values
$\rho=1$, $1.5$, $2$ and~$2.5$ and $\beta$ as above (notice that the data corresponding to~$\rho=2$ have been used in the previous experiment). For the FEM approximation we consider quadratic elements on a uniform grid with $J=80$ elements and compute the stable periodic solutions as described above. We also compute the periodic solutions on a grid with $J=160$ elements so that we can compare the two approximations in order to have an idea of the accuracy of the FEM approximation with $J=80$ elements. We do this on a very fine partition (2048 subdivisions) of the corresponding period $[0,T^{\nu_l,\beta_k}]$ for each pair~$(\nu_l,\beta_k)$, where $\nu_l = 10^{-\rho_l}$. The results can be seen in~Fig.~\ref{fig_dos_par1}, where we can see that the accuracy of the approximations varies considerably from the value pairs $(\nu_l,\beta_k)$  near~$(0.1,2.75)$ to those close to~$(10^{-2.5}, 4.25)$.
\begin{figure}[h]
\begin{center}
\includegraphics[height=5truecm]{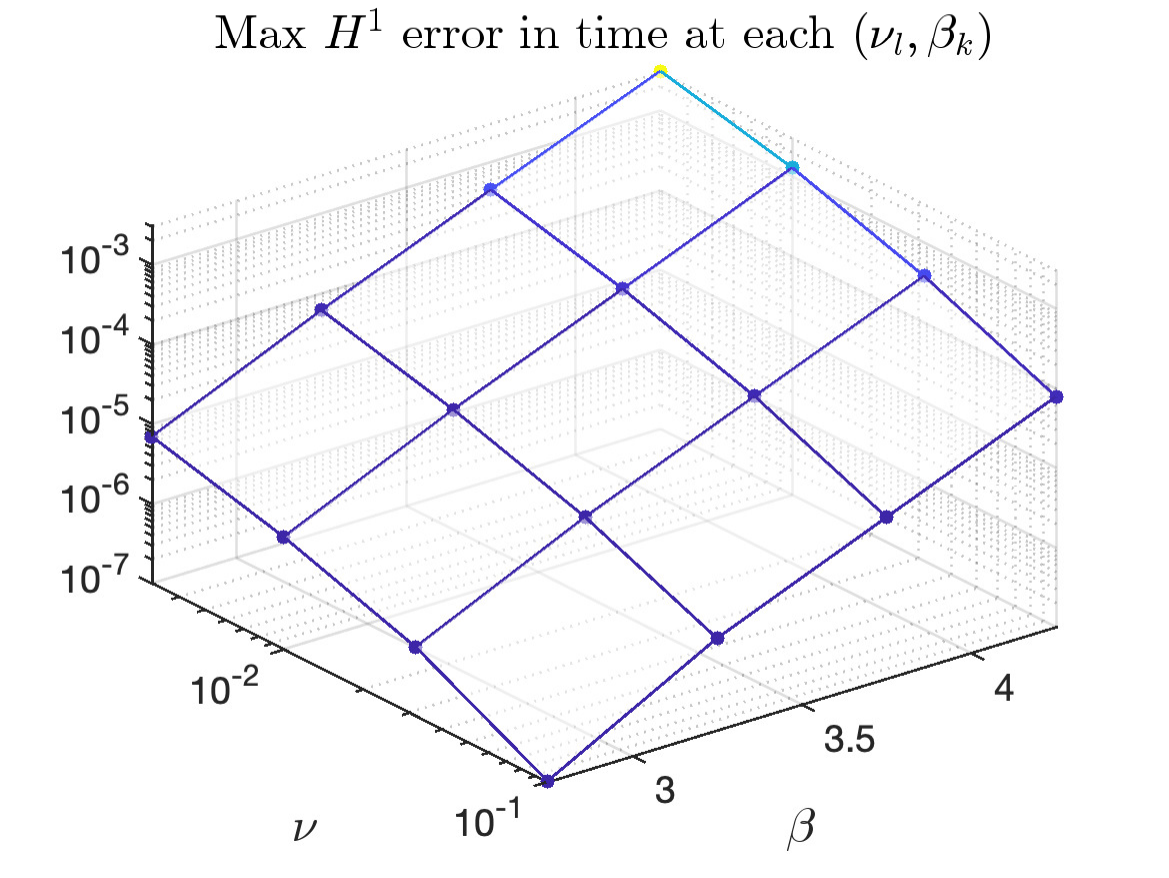}\quad
\includegraphics[height=5truecm]{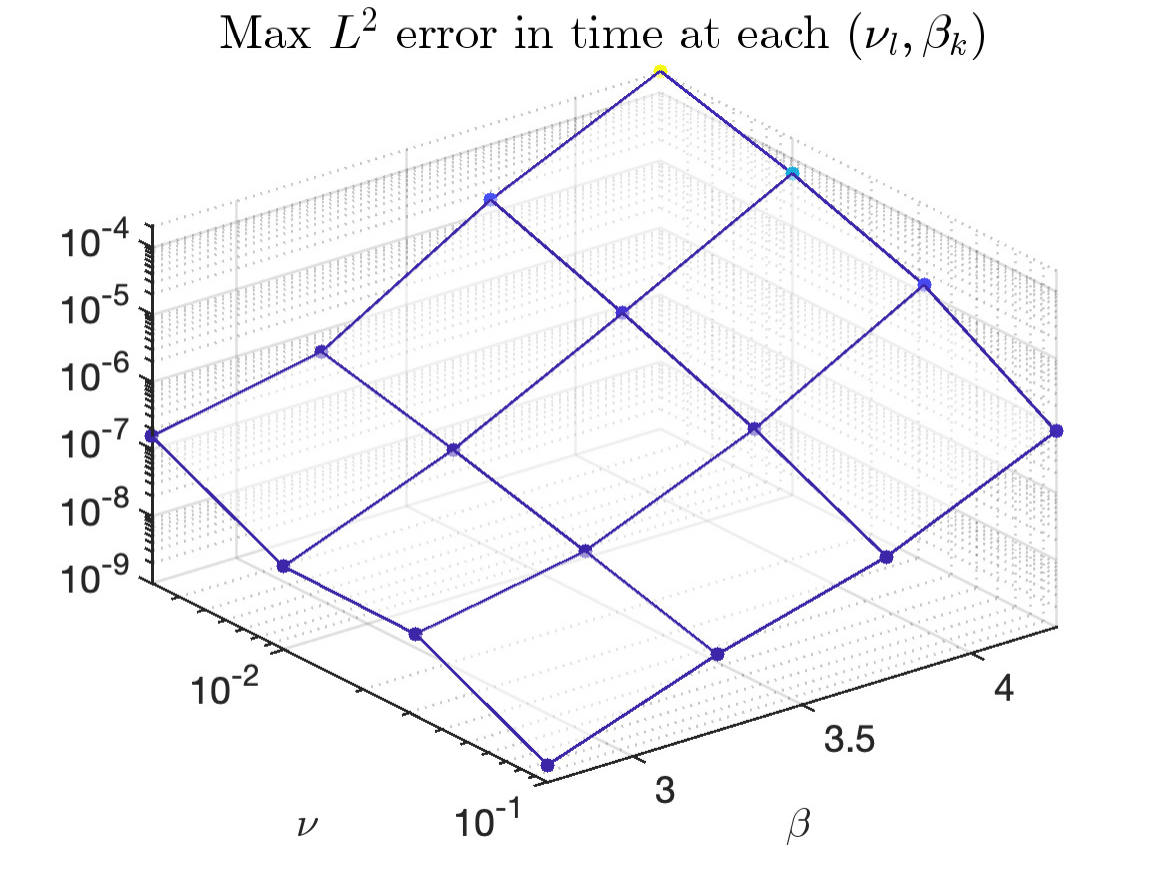}
\caption{\label{fig_dos_par1} Two parameters: accuracy of the FEM approximations used in the data set}
\end{center}
\end{figure}

Next, we find out what values of $M$ (number of snapshots in each period) and~$r$ (dimension of the POD space) should be used so that the POD method matches the accuracy of the FEM approximation. We found that with $M=64$ one has to take $r=80$ in order to have $\max_{t}\| \nabla \be_r(t)\|_0$ below $4\times 10^{-3}$, whereas this was achieved for~$r=65$ if $M=128$ snapshots per period were used. Thus, in the sequel, we use~$M=128$ and~$r=65$. The $H^1$-errors for the value pairs $(\nu_l,\beta_k)$ used in the data set can be seen
in~Fig.~\ref{fig_dos_par2} (left plot), where, as before, we see that they vary in size in more than three orders of magnitude. In discontinuous line, the errors~$\max_{t}\| \nabla {\boldsymbol\varepsilon}_r(t)\|_0$ are also shown. Since it is difficult to see how small they are, we show the ratios~$\max_{t}\| \nabla \be_r(t)\|_0/\max_{t}\| \nabla {\boldsymbol\varepsilon}_r(t)\|_0$  on the right plot, where we can see that all values are above $0.05$. Also, we found that 
ten out of sixteen values are above 0.03.
\begin{figure}[h]
\begin{center}
\includegraphics[height=5truecm]{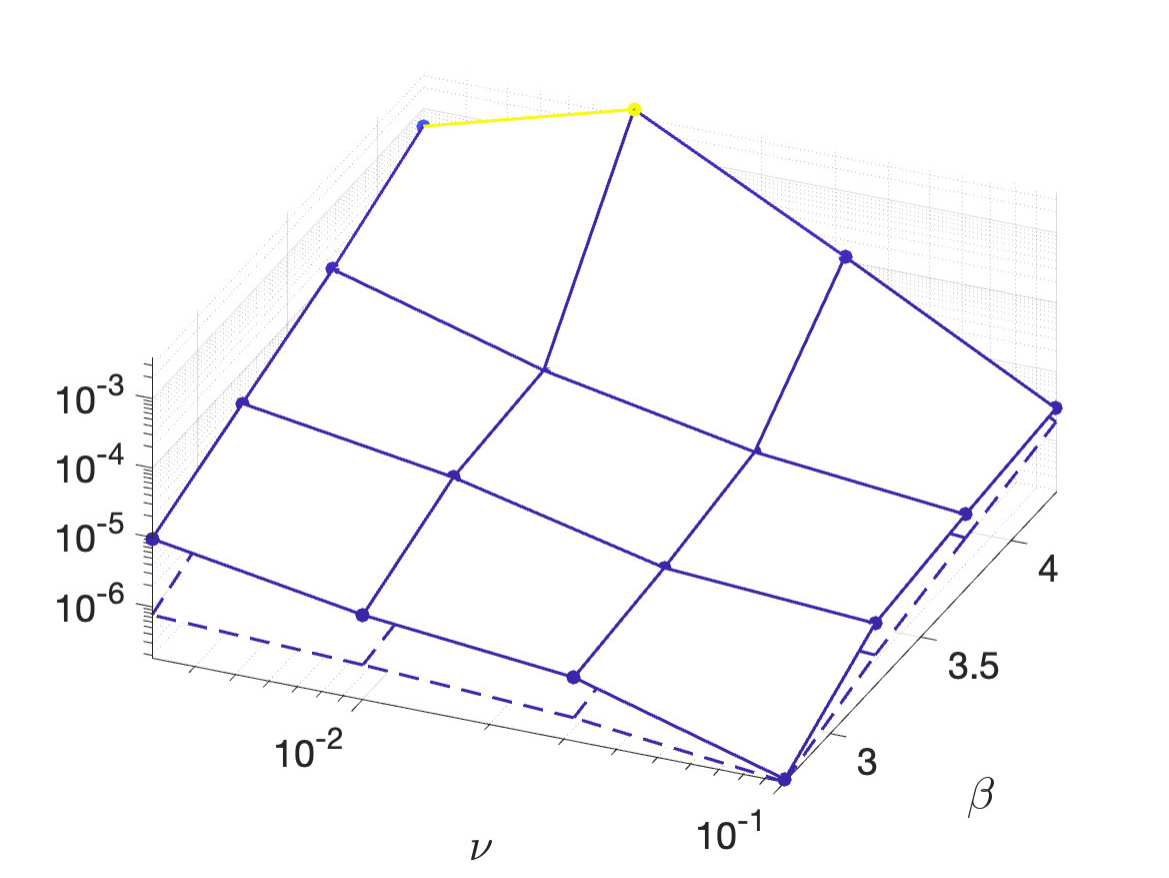}\quad
\includegraphics[height=5truecm]{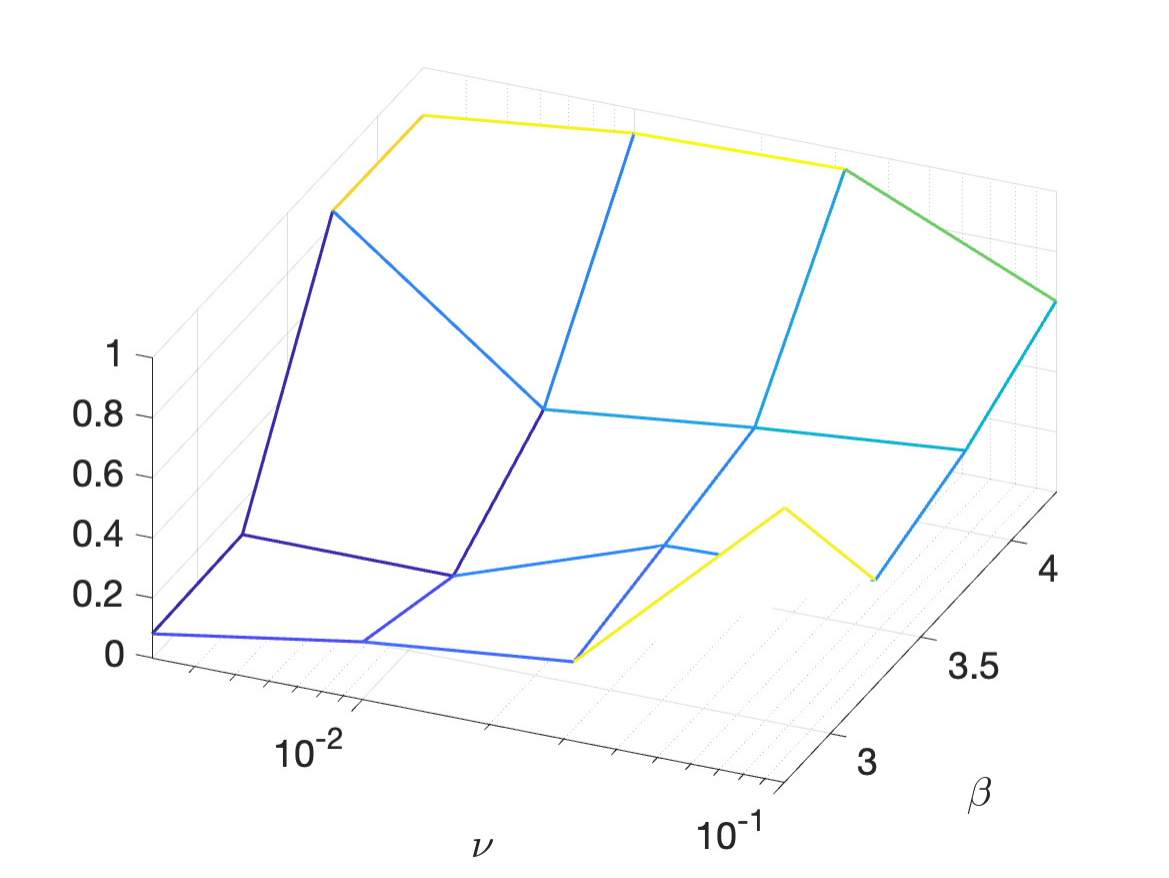}
\caption{\label{fig_dos_par2} Two parameters: Left, $\max_{t\in [0,T^{\nu_l,\beta_k}]}\|\nabla \be_r(t)\|_0$ (continuous line) and  $\max_{t\in [0,T^{\nu_l,\beta_k}]}\| \nabla {\boldsymbol\varepsilon}_r(t)\|_0$ (discontinuous line); Right, ratios $\max_{t\in [0,T^{\nu_l,\beta_k}]}\| \nabla \be_r(t)\|_0/\max_{t\in [0,T^{\nu_l,\beta_k}]}\| \nabla {\boldsymbol\varepsilon}_r(t)\|_0$.}
\end{center} 
\end{figure}

Next, to simulate the POD method outside the data set, we compute the periodic solutions of the POD approximation for value pairs~$(\rho,\beta)$ on a $30\times 30$ uniform partition of the interval~$[1,2.5]\times[2.75,4.25]$. Again, these were computed without any further computation with the FEM method. We also computed the FEM approximations for the same value pairs. The relative error between the periods of both methods was found to be below $2\times 10^{-7}$. Also, for each of these value pairs, we measure the errors $\be_r(t)$ over a fine partition of the interval~$[0,T^{\nu,\beta}]$, $T^{\nu,\beta}$ being the corresponding period. The results can be seen in~Fig.~\ref{fig_dos_par3}, were we show $\max_{t\in [0,T^{\nu,\beta}]}\|\nabla \be_r(t)\|_0$ (left plot) and
$\max_{t\in [0,T^{\nu,\beta}]}\| \be_r(t)\|_0$ (right plot); the red circles correspond to value pairs used in the data set. We see that the errors outside the data set do not differ much from those in the data set. Again, the fact that the errors in different parts of the parameter space differ in several orders of magnitude suggests to use values of~$r$ that depend on the value pairs $(\nu,\beta)$. This will also be subject of future work.
\begin{figure}[h]
\begin{center}
\includegraphics[height=5truecm]{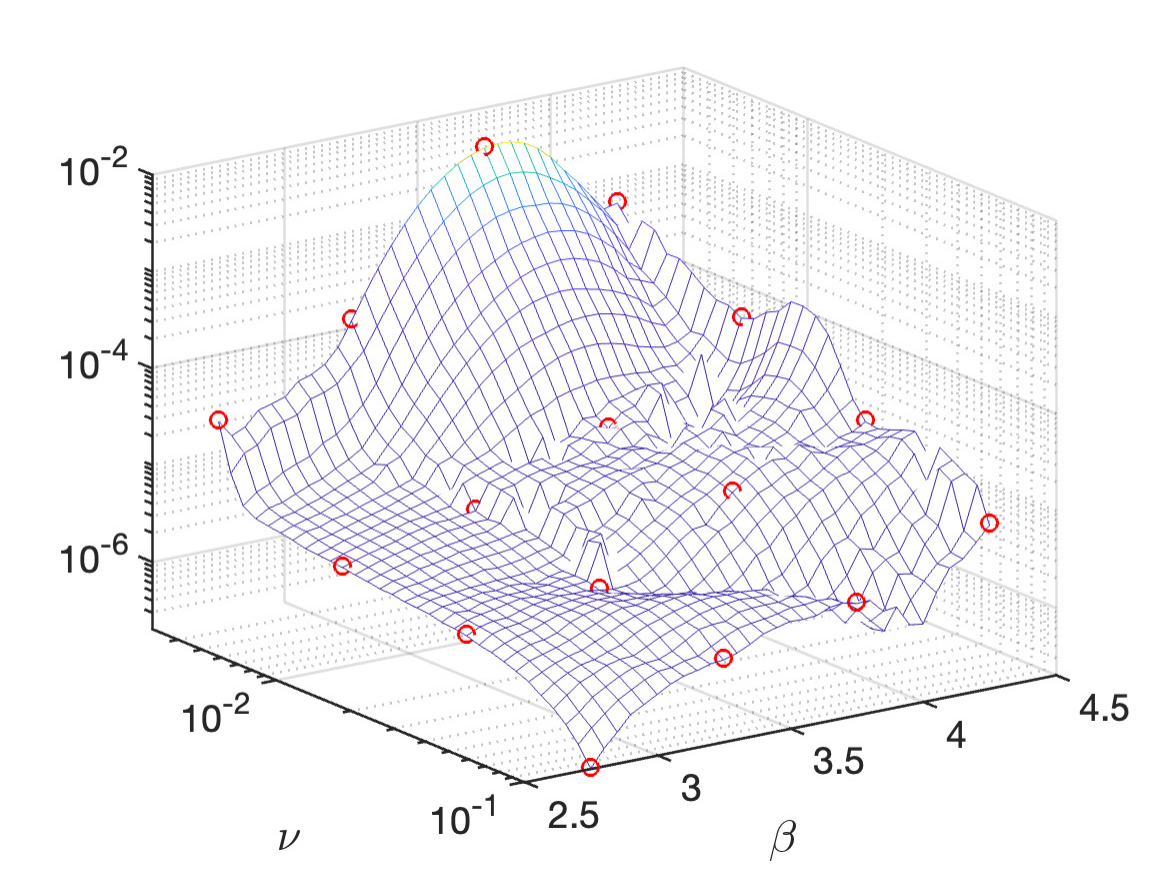}\quad
\includegraphics[height=5truecm]{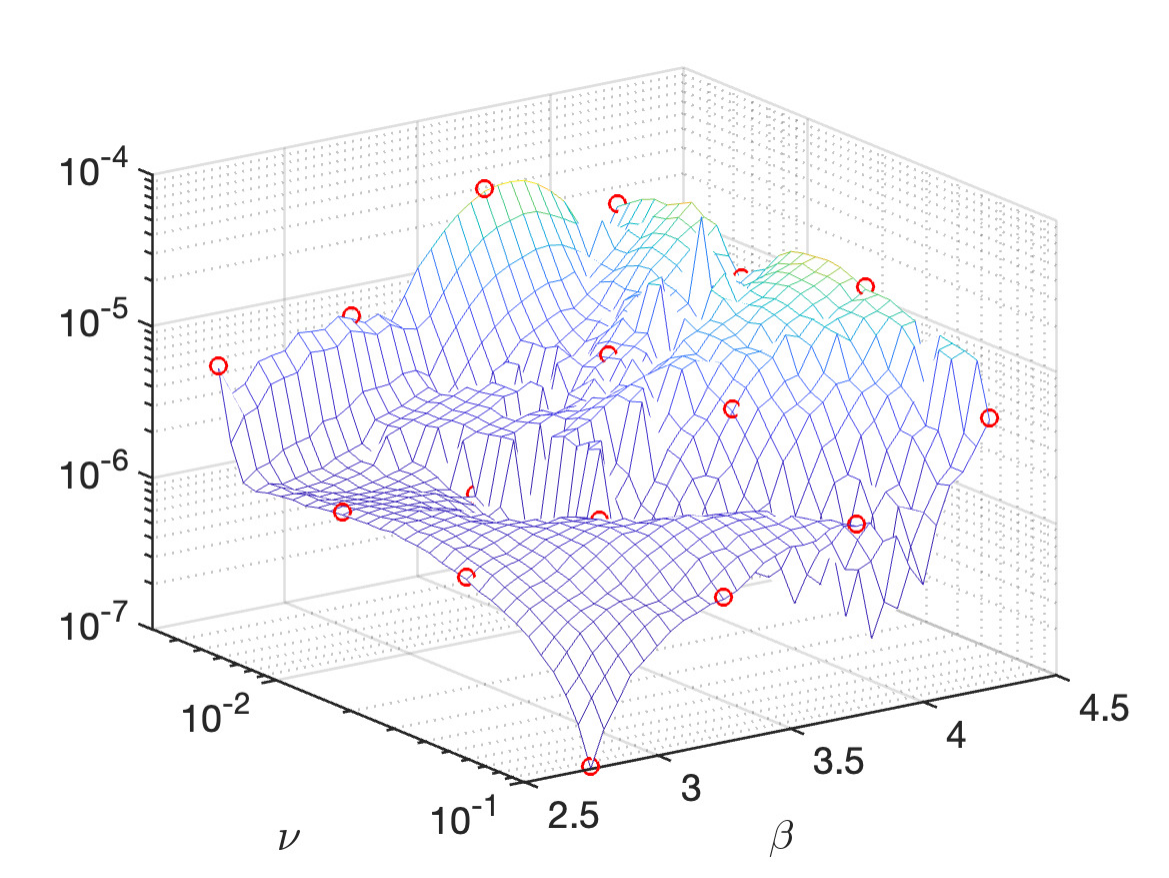}
\caption{\label{fig_dos_par3} Two parameters: Errors $\be_r$ of the POD method in~Section~\ref{sec6} outside the data set; left, $\max_{t\in [0,T^{\nu,\beta}]}\|\nabla \be_r(t)\|_0$, right,
$\max_{t\in [0,T^{\nu,\beta}]}\|\be_r(t)\|_0$. Red circles correspond to value pairs used in the dataset.}
\end{center}
\end{figure}

\section{Conclusions and future work}
In this paper we propose a new POD method for parametric time-dependent reaction-diffusion partial differential equations. The method is based on finite differences (respect to time and parameters) of some snapshots. The snapshots are finite element approximations to the solutions of the PDE at different times and values of parameters in a selected set. The method is designed in such a way that pointwise-in-time estimates can be proved at any time in a given interval. The a priori bounds are also valid for any value of the parameters (including out-of-sample values). The error in the POD approximations depend on the tail of the eigenvalues defined in \eqref{the_sigma_r} and on the distance between two consecutive values of time where the snapshots are taken and the distance between two consecutive values of the parameters. Interpolation techniques are used to achieve the error bounds for out-of-sample values. In the paper we also consider a standard POD method based on the set of snapshots for different values of time and
parameters and prove analogous error bounds using the techniques developed for the new method. For the standard case the bounds are quasi-optimal, as the exponent in the tail of the eigenvalues depends on the smoothness of the finite element approximations. The smoother (with respect to time and parameters) the approximations are, the closer the exponent gets to the optimal value 1.

We present some numerical experiments to show the performance of the new method in practice. We take a reaction-diffusion system depending on some parameters and consider the case of only one parameter apart from time and a two-parameter case. For the first case, we compare the results of the new method with the standard one. In the numerical experiments shown in the present paper the new method improves the performance of the standard one, although for smaller values of $r$ (i.e., lower accuracy) we found that the  traditional method produced slightly smaller errors. In both cases, in agreement with the theory, accurate approximations are obtained for any value of time and any value of the parameters in some given intervals. We also show that the POD method (both new and standard) is able to compute periodic orbits of the reaction-diffusion system (both period and initial condition) for out-of-sample values of the parameters without any additional information. The  accuracy of the POD method proved to be excellent in the computation of the periods with a relative error respect to the periods of the FEM approximation around $10^{-9}$.

Our experiments are restricted to equally spaced-times and parameters. As a subject of future research we will study the case in which we concentrate snapshots in those parts with larger variations. In our experiments the errors in different parts of the parameter space differ in several orders of magnitude which suggests that different values of $r$ could be considered for different values of the parameters. This will be also subject of future research since we have studied the variation of the errors with $r$  but keeping always the same value of $r$ for the given set of parameters.  
\bibliographystyle{abbrv} 

\end{document}